\documentclass[12pt]{iopart}
\usepackage{graphicx}              
\usepackage{iopams}
\expandafter\let\csname equation*\endcsname\relax 		
\expandafter\let\csname endequation*\endcsname\relax 		
\usepackage{amsmath}               
\usepackage{amsthm}                

\usepackage{mathptmx}      
\usepackage{inputenc}
\usepackage{amsmath}               
\usepackage{amssymb}
\usepackage{latexsym}

\usepackage{algorithm}
\usepackage{algorithmic}

\usepackage{tikz}
\usetikzlibrary{matrix}

\newtheorem{thm}{Theorem}[section]
\newtheorem{lem}[thm]{Lemma}

\newtheorem{prop}[thm]{Proposition}

\newtheorem{mydef}[thm]{Definition}
\newtheorem{rem}[thm]{Remark}

\newtheorem{assumption}[thm]{Assumption}
\usepackage{mathptmx}      
\usepackage{inputenc}
\usepackage{amsmath}               
\usepackage{amssymb}
\usepackage{latexsym}

\usepackage{graphics,color}

\newcommand\amf[1]{{\color{black}{#1}}}

\begin{document}
\title[Solving an inverse problem using conjugation]
     {Solving the inverse problem for an ordinary differential equation using conjugation}
\author{Alfaro Vigo, D. G.$^1$, \'Alvarez, A. C.$^1$, Chapiro, G.$^3$, Garc\'ia, G.C.$^2$  and Moreira, C. G. T. A.$^4$}
\address{$^1$ Departamento de Ci\^encia da Computa\c{c}\~ao, Universidade Federal do Rio de Janeiro, Caixa Postal
	68.530, CEP 21941-590, Rio de Janeiro, RJ, Brazil.}
\address{$^2$ Departamento de Matem\'atica y Ciencia de la Computaci\'on, Universidad de Santiago de Chile, Casilla 307, Correo 2, Santiago de Chile, Chile.}
\address{$^3$ Departamento de Matem\'atica, Universidade Federal de Juiz de Fora, Juiz de Fora, MG 36036-900, Brazil.}

\address{$^4$ IMPA, Rio de Janeiro, RJ 22460-320, Brazil.}

\ead{dgalfaro@dcc.ufrj.br, amaury@dcc.ufrj.br, grigori@ice.ufjf.br,\\galina.garcia@usach.cl, gugu@impa.br}

\begin{abstract}


{We consider the following inverse problem for an ordinary differential equation (ODE): given a set of data points $P=\{(t_i,x_i),\; i=1,\dots,N\}$, find an ODE $x^\prime(t) = v (x)$ 
that admits a solution $x(t)$ such that $x_i \approx x(t_i)$ as closely as possible. The key to the proposed method is to find approximations of the recursive or discrete propagation function $D(x)$ from the given data set. Afterwards, we determine the field $v(x)$, using the conjugate map defined by Schr\"{o}der's equation and the solution of a related Julia's equation. Moreover, our approach also works for the inverse problems where one has to determine an ODE from multiple sets of data points. 
We also study existence, uniqueness, stability and other properties of the recovered field $v(x)$. Finally, we present several numerical methods for the approximation of the field $v(x)$ and provide some illustrative examples of the application of these methods.}

\end{abstract}

\vspace{2pc}
\\ 
\maketitle


\section{Introduction}

In this paper we solve the problem for determining an ordinary differential equation
(ODE) from given data using conjugacy methods. 
{
In that sense, from a set of data points $(t_i,x_i)$, $i=1,\dots,N$, we find an ODE $x^\prime(t) = v (x(t))$ 
that admits a solution $x(t)$ such that $x_i \approx x(t_i)$ as closely as possible.  
More generally, given multiple sets of data points $P^{(k)}=\{(t^{(k)}_i,x^{(k)}_i),\; i=1,\dots,N^{(k)}\}$, $k=1,\dots,M$, we obtain an ODE admitting solutions $x^{(k)}(t)$ such that $x^{(k)}_i \approx x^{(k)}(t_i)$ as closely as possible.  
 
The key to the proposed method consists in finding an approximation of the discrete propagation function $D$ from the given data. Afterwards, we determine the field $v(x)$, using the conjugate map defined by Schr\"{o}der's equation and the solution of a related Julia's equation.}

{Explicit solutions of Schr\"{o}der's equation can be obtained using analytical and asymptotic methods (see for instance \cite{curtright2011approximate,curtright2009evolution,curtright2010chaotic,kucs2}). However, it is not an easy task, in general. Even when if an asymptotic procedure could be completed, it relies on the knowledge of function $D$, therefore it can fail when the function $D$ is known only approximately. In this paper, we developed numerical procedures to approximate the solution of Schr\"{o}der's equation in the general case.
}

{Methods for the identification of the iteration (or discrete propagation) function $D$ in a discrete dynamical system are widely known (see for instance \cite{brunton2016discovering,nelson2005nonlinear,roy2017dynamic,zhang2006discrete}).
	Our work is fundamentally based on the fact that the iteration function can be determined from the solution data set at some points. As many of the known methods are applicable, we do not delve into that direction, but propose instead a method based on interpolation that has successfully worked in the cases studied here.}

 We select this method as it has been employed to solve  the  Schr\"{o}der's equation  by exact or asymptotic methods \cite{kucs1,kucs2,curtright2009evolution,curtright2011renormalization}, and the Julia's equation by numerical methods \cite{alvarez2005fast,alva13,small2007functional}. The conjugacy method is applicable in this problem
as it allows for establishing a direct and stable relationship between the function $D$ leaving the solution from the time $t$ to $t+1$ and the field of the ordinary differential equation.

The problem  of searching the field knowing 
the solution at some time $t_i$, $i=1,\dots,n$ is studied
in \cite{kunze1999solving,kunze2009using}, where authors  determine an optimal
vector field associated with contractive Picard operators.
 In \cite{kunze2012solving}, using a more general framework, such inverse problems are solved  in another direction using the Collage Theorem. In both formulations the inverse problem is solved and a recovery method for the solution is proposed.
 Regarding  the inverse problem studied here, the following open problem is proposed in \cite{curtright2009evolution}: Can the conjugacy method be applied in all situations? In this spirit, we solve such a problem by proposing a robust numerical technique to solve a functional equation for a class of functions broader than other previously established. Moreover, we show stability results and
 determine  sufficient conditions on data input function to get a priori information of the solution, such as its regularity and monotonicity. 
 
We provide analytical and numerical procedures for the recovery of the field, taking as an input the solution at some points uniformly distributed in time. To this end, we use the conjugacy method, which consists in  solving functional equations (see \cite{kucs1}).
This method was successfully used in \cite{curtright2009evolution,curtright2011renormalization} to determine the continuous behavior of iterates from the lattice of time points. Similar methods are used in other
applications such as the ones in \cite{heard1975change,keller1956determination,miller1969wkb}.

System identification can be defined as the problem of estimating the best possible model of a system, given a set of experimental data, see \cite{pachter2000identification,peifer2007parameter}. The problem studied here can be related  to the identification problem, in the sense that if the field is parametrized then the inverse problem reduces to determine the parameters from the solution at some times $t_i$, $i=1,\dots,n$. The latter has several applications in physics, chemistry, and biology. In this paper, we present an alternative method for identification of the propagation function in an ordinary differential equation by using the conjugacy method. 

\subsection{The conjugacy method}

Below we present      a brief introduction to the conjugacy method and
how it is related to the problem of  field recovery in an
ordinary differential equation.

Consider an evolution trajectory $x(t)$ of a 1-dimensional system specified by a local, time-
translation-invariant law given by the ordinary differential equation (ODE)
\begin{equation}
dx/dt = v(x).
\label{eq1}
\end{equation}
Equation \eqref{eq1} can be integrated to produce the trajectory as a family of functions of the initial data
indexed by the time
\begin{equation}
x(t)=f_t(x(0)),
\label{eq2}
\end{equation}
where $f_t$ is a group of invariant transformations $f_t$ (see \cite{arnold2012geometrical}).

Thus, for a given time $t$ and increment $\Delta t$ (here we take for easy notation $\Delta t=1$, but the same ideas can be applied  for any $\Delta t$), we have
\begin{equation}
x(1)=f_1(x(0)),
\label{eq3}
\end{equation}
For the case of any time $t$, we have

{
\begin{equation}
x(t + 1) = f_{t+1}(x(0)) = f_1(x(t)) = D(x(t)),
\label{eq4}
\end{equation}}
i.e., $x(t + 1)$ is the same function of $x(t)$ as $x(1)$ is of $x(0)$. 
{Here, $D = f_1$ denotes the unit-time discrete propagation (or iteration) function associated with the ODE in Eq. \eqref{eq1}. For notational convenience, we shall often use $x(0) \cong x$.}

Using the above framework we can state the following problems.

\textit{Problem 1: How does one obtain the complete continuous trajectory $x(t)=f_t(x)$ knowing the discrete propagation function $D(x) = f_1(x)$ in  \eqref{eq4}?}

\textit{Problem 1} was solved in \cite{curtright2009evolution} for certain kind of analytical functions $D$ under some hypotheses covering  physical problems of great interest. Additionally, notice that if this problem is solved we can obtain the corresponding field $v(x)$ of the ODE \eqref{eq1} by the formula
\begin{equation}
v(x)=\frac{\partial }{\partial t} f_t(x)|_{t=0}.
\label{eq5}
\end{equation}

In summary, the method described in \cite{curtright2009evolution} to solve  \textit{Problem 1} consists of the following: Given the differentiable function $D: [0,a] \rightarrow [0,b]$, that represents the discrete evolution function in \eqref{eq4}, we can find a diffeomorphism $h$, which satisfies the conjugation property (Schr\"{o}der's equation)
\begin{equation}
h(D(x))=s\,h(x),\label{sch}
\end{equation}
for some constant $s \neq 1$. The existence, uniqueness and regularity of the solutions of this equation were widely addressed in the literature \cite{kucs1,kucs2}. This equation has been used in some applications, e.g., \cite{curtright2009evolution, heard1975change}. It has the following property: when the origin is a fixed point of $D$ (i.e., $D(0)=0$), it follows that $h(0) = 0$ and, if $h'(0) \neq 0$ then $s = D'(0)$.
It is easy to calculate the integer iterates of $D$: $h(D^n(x))=s^nh(x)$ and $D^n(x)=h^{-1}(s^n h(x))$, where $D^n(x)=D(D^{n-1}(x))$ represents the iterates of $D$.

The iteration property of equation \eqref{sch} can be extended to any real $t$ by considering
\begin{equation}
x(t)=f_t(x(0))=D^t(x(0))=h^{-1}(s^t h(x(0))),\label{contia1}
\end{equation}
where $D^t$ is an analytic function. Thus, if we have the function $h$, the iterates  of the function $D$ can be obtained for any real $t$. This method introduces a change of variables leading to a better representation of the trajectories in the neighborhood of the fixed point $x=0$.

The velocity of the flux in \eqref{contia1} can be obtained as 
\begin{equation}
v(x(t))=d D^t(x)/ dt = \log(s) h(x(t))/h'(x(t)).
\label{eqref1}
\end{equation}
Schematically, the method to obtain the field $v$ from $f_1$ can be represented as
\begin{equation}
\{D\} \rightarrow h \rightarrow v
\end{equation}

The method described above  was used in \cite{curtright2009evolution}  for different cases in which  the field can be obtained analytically in closed form.

Using solution of \textit{Problem 1}, we can solve the following problem.

\textit{Problem 2: 
	How does one obtain the field $v(x)$ from a solution of the ODE \eqref{eq1} given at  discrete  times $t_i$, $i=1,\dots,N$, i.e.,  given $x(t_i)$ for $i=1,\dots,N$?}

To solve  \textit{Problem 2}, we use the solution of \textit{Problem 1}. To do so, the iterate function $D$ in  \eqref{eq3} is obtained from a solution of the ODE \eqref{eq1} at discrete times $t_i$, i.e., from the data $x(t_i)$, with $i=1,\dots,N$.
This strategy is based on the fact that for times $t_i$ uniformly spaced (i.e., $t_{i+1}-t_i=\Delta t$), we have $x(t_{i+1})=D(x(t_i))$, therefore function $D$ can be extended for all  $x$ knowing its values at points $x(t_i)$.
Insomuch as if we have the continuous trajectory $x(t)$ we can obtain the field $v(x)$ by the formula \eqref{eqref1}.
Schematically, this strategy to obtain the field from a given data can be represented as
\begin{equation}\label{eq:strat_1}
\{x(t_i), i=1,\cdots,N\} \rightarrow \{D\} \rightarrow h \rightarrow v.
\end{equation}

Denoting $g(x)= h(x)/h'(x)$, equation \eqref{eqref1} can be rewritten as
\begin{equation}
v(x(t))=d D^t(x)/ dt = \log(s) g(x(t)).
\label{eqref1a}
\end{equation}
Moreover, one can verify that the function $g$ satisfies  Julia's equation
\begin{equation}
g(D(x))=D'(x)g(x),
\label{jul1}
\end{equation}	
with the condition $g^\prime(0)=1$.
Thus, alternatively, obtaining the function $g$  from equation \eqref{jul1}, without relying on the computation of $h$ from equation \eqref{sch}, leads to another strategy  to recover the field $v$, which can be represented schematically as
\begin{equation} \label{eq:strat_2}
\{x(t_i), i=1,\cdots,N\} \rightarrow \{D\} \rightarrow g \rightarrow v.
\end{equation}

This paper focuses on the application of this second strategy.

\subsection{Organization of the paper}

This paper is organized as follows.

In Section \ref{sec:exist}, we discuss the  existence  and uniqueness of the solution of the  Schr\"{o}der's and Julia's equations for a certain class of functions $D$.
The proof is straightforward and it includes the basics elements to approximate the solution of the functional equation in the general case. In Section \ref{secU}, we  provide  some properties of the solutions of Julia's equation.
Section \ref{sec2} presents  the sensitivity stability of two fields for close initial functions $D_1$ and $D_2$ in \eqref{eq3}.

In Section \ref{sec1}, we describe several approximate methods for the solution of \textit{Problem 2}.  These methods are based on a combination of numerical techniques and analytical results.  
In Section \ref{sec:examples} we present several numerical examples that use the proposed approximate methods to recover the field $v(x)$, and illustrate their main characteristics. 

Finally, section \ref{rec:final_r} is dedicated to present the final  remarks.

{We suggest that readers who are more interested in the approximate methods and the numerical examples could skip the technicalities given in Sections \ref{sec:exist}--\ref{sec2} and continue reading from Section \ref{sec1}.
}

\section{Existence and uniqueness of solutions} 
\label{sec:exist}

We reduce the inverse problem of recovering the field of the ordinary differential equation to obtain the solution through
either Schr\"{o}der's or Julia's equation. The construction of the function $D$ in \eqref{eq3} from the data $x(t_i)$, with $
i=1,\dots,N$ is provided in Section \ref{sec:recover_D}. The error in this calculation corresponds to the interpolation error involved in the procedure.

In this section we prove some theoretical results of the existence  and uniqueness of the solution for Schr\"{o}der's (equation \eqref{sch}) and Julia's equations (equation \eqref{jul1}) given the function $D$ in \eqref{eq4}.

We assume that the function $D$ has an attractive  fixed point at  $x=0$ and $|D(x)| < |x|$ in a neighborhood of $x=0$. These assumptions are not as restrictive because the equation can often be transformed, e.g., by substituting $y=D^{-1}(x)$ (see \cite{small2007functional}). The same results can be adapted for any fixed point different from zero or more general if $D$ has several fixed points.
	
	 We use here results from \cite{alva13,kucs2,small2007functional} for the solution of \eqref{sch} and \eqref{jul1};
however, we also obtain other properties that will be used to prove the stability results. 

\begin{mydef}
	A local diffeomorphism $h:(-a,a)\to(-c,d)$, with $a,c,d>0$, such that $D(x)=h^{-1}(\lambda h(x))$, $\forall x \in(-a,a)$ is called a {\bf conjugation} between function $D$ and its linear part $L(x)=\lambda x$, where $\lambda$ is a positive parameter.
	\label{def:1}
\end{mydef}

\amf{The} next proposition proves the existence of a conjugation diffeomorphism for \amf{a} certain class of functions $D$, which proves the existence and uniqueness of
the solution of Schr\"{o}der's equation.

\begin{prop}\label{prop:1} 
	Assume $a>0$ and $D:(-a,a)\to\mathbb{R}$, $D\in C^{1+\epsilon}$ for some $\epsilon>0$ 
	such that (a) $D(0)=0$, (b) $0<\lambda:=D'(0)<1$ and (c) $|D(x)|<|x|$, 
	$\forall x \in(-a,a)\backslash\{0\}$. 
	Then there is a unique conjugation $h$ as in Definition~\ref{def:1}, such that $h'(0)=1$. 
	\amf{This} diffeomorphism $h$ is of
	class $C^{1+\epsilon}$ and its derivative is given by 
	\begin{equation}
	h'(x) = \prod_{i=0}^{\infty}\frac{D'(D^i(x))}{\lambda}, 
	\label{eq:hlinhahj}
	\end{equation}
	where $D^0(x)=x$ and $D^{i+1}(x)=D(D^i(x))$, $\forall i\geq 0$. 
\end{prop}

\begin{rem}
	\label{rem:1}
	Under the hypotheses of the above proposition, taking smaller $a$, if necessary, there 
	are $\bar b$ and $b$ such that $0<\bar b < D'(x) < b < 1$, $\forall x \in(-a,a)$ 
	and $|D(x)|\leq b|x|<|x|$, 
	$\forall x \in(-a,a)\backslash\{0\}$. 
\end{rem}

\begin{proof} 
	Since the function $h$ satisfies $h(D(x))=\lambda h(x)$ and, for $x=0$, 
	$h(0)=\lambda h(0)$, it follows that $h(0)=0$. Moreover,
	$h(D^n(x)) = \lambda^n h(x)$, $\forall n\in\mathbb{N}$, yielding $h'(D^n(x)) (D^n)'(x)=\lambda^n h'(x)$, 
	and thus 
	\begin{equation}
	h'(x)=
	h'(D^n(x)) \frac{(D^n)'(x)}{\lambda^n}=
	h'(D^n(x)) \prod_{j=0}^{n-1} \frac{D'(D^j(x))}{\lambda}.
	\label{eq:aux3}
	\end{equation}
	Using Remark~\ref{rem:1} we get $|D^n(x)|\leq b^n|x| \to 0$, when $n\to\infty$. 
	Since $h'(0)=1$ and $h\in C^{1}$ it follows that $h'(x)$ satisfies equation \eqref{eq:hlinhahj}. 
	This implies the uniqueness of $h$.
	
	Next we prove that the product in equation 
	(\ref{eq:hlinhahj}) converges. \amf{This} is equivalent to \amf{proving} that 
	\begin{equation}
	f(x) = \sum_{j=0}^{\infty} (\log D'(D^j(x))-\log \lambda)
	\label{eq:aux4}
	\end{equation}
	converges and so the function $f$ is well defined. Since $D'\in C^{\epsilon}$, it follows 
	that $\log D'\in C^{\epsilon}$\amf{,} thus there exists $k>0$ such that 
	$|\log D'(z) - \log D'(y)| \leq k |z-y|^\epsilon$, $\forall y,z\in(-a,a)$. In particular, 
	$|\log D'(z) - \log \lambda| \leq k |z|^\epsilon$, $\forall z\in(-a,a)$.
	
	Since $|D^{j}(x)| \leq b^j |x| \leq a\cdot b^j$, $\forall x \in(-a,a)$, 
	$|\log D'(D^j(x))-\log\lambda| \leq k |D^j(x)|^\epsilon \leq k \cdot a^\epsilon\cdot b^{j\epsilon}$, 
	$\forall j\geq 0$, yielding the absolute and uniform convergence of the
	series in equation (\ref{eq:aux4}).
	Moreover, since $|D^j(x)-D^j(y)|=|(D^j)'(\xi)| |x-y|$, for some $\xi\in (x,y)$, and 
	$|(D^j)'(\xi)|=\prod_{j=0}^{j-1} |D'(D^j(\xi))| \leq b^j$, we have, for every $x, y\in (-a,a)$,
	\begin{align}
	|f(y)-f(x)| 
	& \leq    \sum_{j=0}^{\infty} | \log D'(D^j(y)) - \log D'(D^j(x))| \\
	& \leq  k \sum_{j=0}^{\infty} |D^j(y)-D^j(x)|^\epsilon \\
	& \leq  k \sum_{j=0}^{\infty} b^{j \epsilon} |x-y|^\epsilon = \frac{k}{1-b^\epsilon} |x-y|^\epsilon, 
	\end{align}
	yielding $f\in C^{\epsilon}$. Thus,
	\begin{equation}
	e^{f(x)} = \prod_{j=0}^{\infty} \frac{D'(D^j(x))}{\lambda}
	\end{equation}
	is also of class $C^\epsilon$. Notice that 
	\begin{equation}
	\begin{array}{cl}
	\displaystyle
	e^{f(D(x))} & =  \displaystyle \prod_{j=0}^{\infty} \frac{D'(D^j(D(x)))}{\lambda} 
	= \prod_{j=1}^{\infty} \frac{D'(D^j(x))}{\lambda} 
	= \frac{\lambda}{D'(x)} \prod_{j=0}^{\infty} \frac{D'(D^j(x))}{\lambda} \\
	& =  \displaystyle \frac{\lambda}{D'(x)} e^{f(x)}. 
	\end{array}
	\end{equation}
	Defining
	\begin{equation}
	h(x) = \int_{0}^{x} e^{f(t)} dt, 
	\label{eq:func1}
	\end{equation}
	we have $h(0)=0$, $h'(x)= e^{f(x)}$, and thus
	$
	(h(D(x)))' = h'(D(x)) D'(x) = $ \\ $e^{f(D(x))} D'(x) =
	\lambda e^{f(x)} = \lambda h'(x)$, $\forall x \in (-a,a).
	$
	Since $h(D(0))=h(0)=0=\lambda h(0)$, we have $h(D(x))=\lambda h(x), \forall x \in (-a,a)$.
	Since $h'(x)=e^{f(x)}>0$, $\forall x \in(-a,a)$, $h$ is a 
	diffeomorphism over its image $(-c,d)$. From $h'(x)= e^{f(x)}\in C^\epsilon$ it follows that
	$h\in C^{1+\epsilon}$. Since 
	$h(D(x)) = \lambda h(x)$, we have $D(x) = h^{-1}(\lambda h(x))$, $\forall x \in(-a,a)$. Finally, notice that
	\begin{equation}
	f(0) = \sum_{j=0}^{\infty}(\log(D'(D^j(0))) - \log\lambda) = 0, 
	\end{equation}
	yielding $h'(0) = e^{f(0)}=1$.
\end{proof}

\begin{rem}
	\label{rem:1a}
	If in the hypothesis of the above proposition, we choose $D$ defined on the interval 
	$[0,M]$, the proposition is still valid in the case when function $D$ is 
	differentiable at zero.
\end{rem}

\begin{rem}
	$\bullet$ If $b\neq 0$, the unique conjugation $\tilde h$ of class $C^1$ between $D$ and its linear part with
	$\tilde h'(0)=b$ is defined by $\tilde h(x)=bh(x)$.\\
	$\bullet$ There exists a function $D$ of the class $C^1$, such that the product in 
	equation (\ref{eq:hlinhahj}) does not converge, 
	$\forall x \neq 0$, and a conjugation $h$ of class $C^1$ in the sense of 
	Definition~\ref{def:1} does not exist. 
\end{rem}

\begin{mydef}
	Two functions $f(w)$ and $\tilde f(w)$ are {\bf similar} (denoted by $f(w)\sim\tilde f(w)$) when 
	\begin{equation}
	\lim_{w\to 0}\frac{f(w)}{\tilde f(w)}=1. 
	\end{equation}
\end{mydef}

\begin{rem}
	\label{remer}
	The function $h$ of the Proposition~\ref{prop:1} is such that 
	$\forall n\in\mathbb{N}$, $D^n(x)=h^{-1}(\lambda^n h(x))$. From $h'(0)=1$ it follows
	that $h^{-1}(y)\sim y$, yielding $D^n(x)\sim\lambda^n h(x)$. Thus $h(x)$ can be obtained through the expression 
	\begin{equation}
	h(x)=\lim_{n\to\infty}{\lambda^{-n}D^n(x)}.
	\label{eq:for1}
	\end{equation}
This equation gives a relatively simple method for obtaining the solution of Schr\"{o}der's equation, which is called Koenigs algorithm (see \cite{small2007functional}). Notice that in this algorithm the computation of the derivative of $D$ is not necessary.
\end{rem}




Given the functions $D$ and $h$ as in Proposition~\ref{prop:1} it \amf{may} be checked that the function
$g:=h/h'$ satisfies Julia's equation studied in \cite{kucs2}:
\begin{equation}
g(D(x)) = D'(x)g(x).
\label{eq:func_eq}
\end{equation}
Moreover, $g(0)=0$ and $g$ is differentiable at $0$ with $g^\prime(0)=1$. Indeed, since $h(x)\sim x$ and $h^\prime(x)\sim 1$, it follows that $g(x)\sim x$. Consequently, $g(0)=0$ and $g^\prime(0) = 1$.

\begin{prop}\label{prop:new}
\amf{Assume} that $D$ satisfies the conditions of Proposition~\ref{prop:1} and $\tilde g:(-a,a)\to\mathbb{R}$ is a solution of Julia's equation \eqref{eq:func_eq} \amf{differentiable} at $0$, then it is given by  
\begin{equation}
\tilde g(x) = \tilde{g}'(0) \frac{h(x)}{h'(x)},
\label{eq62}
\end{equation}
where $h$ is the solution of Schr\"{o}der equation.
\end{prop}

\begin{proof}
First, notice that $\forall n \geq 1$, $\tilde g(D^n(x))=(D^n)'(x)\tilde g(x)$. We prove it  using induction. For $n=1$, this is the initial functional equation. Assuming \amf{that} the equation \amf{is} valid for $n$, we have $\tilde g(D^{n+1}(x)) = \tilde g(D(D^n(x))) = D'(D^n(x)) \tilde g(D^n(x)) = D'(D^n(x))(D^n)'(x) \tilde g(x) = (D\circ D^n)'(x) \tilde g(x) = (D^{n+1})'(x) \tilde g(x)$.

From $\lim_{n\to\infty}{D^n(x)}=0$ it follows that $\tilde g(D^n(x))\sim \tilde g'(0)D^n(x)$. From $D(x) = h^{-1}(\lambda h(x))$ it follows that
$$
D^n(x)=h^{-1}(\lambda^n h(x)) \text{ and }
(D^n)'(x)=(h^{-1})'(\lambda^n h(x)) \lambda^n h'(x) \sim \lambda^n h'(x)
$$
(since $h'(0)=1$), so 
$$\tilde g(D^n(x))\sim \tilde g'(0) D^n(x) = \tilde g'(0) h^{-1}(\lambda^n h(x)) 
\sim \tilde g'(0) \lambda^n h(x),$$ 
and thus
$$\tilde g'(0) \lambda^n h(x) \sim (D^n)'(x) \tilde g(x) \sim \lambda^n h'(x) \tilde g(x).$$
This yields
\begin{equation}
\tilde g(x) = \tilde{g}'(0) \frac{h(x)}{h'(x)}.
\end{equation}

Next, we prove that for all $\tilde b\in\mathbb{R}$, the function $\tilde g(x)=\tilde b g(x) =\tilde b {h(x)}/{h'(x)}$ is a solution of the functional equation (\ref{eq:func_eq}). 
We have 
\begin{equation}
\tilde g(D(x))= \frac{\tilde b h(D(x))}{h'(D(x))}. 
\label{eq:aux1}
\end{equation}
From $h(D(x))=\lambda h(x)$ follows $h'(D(x)) D'(x) = \lambda h'(x)$ and 
\begin{equation}
\frac{h(D(x))}{h'(D(x))} = \frac{\lambda h(x) D'(x)}{\lambda h'(x)} = \frac{h(x)}{h'(x)} D'(x).
\label{eq:aux2}
\end{equation}
Substituting \eqref{eq:aux2} into \eqref{eq:aux1} yields 
$\tilde g(D(x)) = \tilde b h(x) D'(x) / h'(x) =  D'(x) \tilde g(x)$. Thus $\tilde g$ in \eqref{eq62} satisfies the 
functional equation (\ref{eq:func_eq}).
\end{proof}

\subsection{Explicit solution of Julia's equation}
\label{num1}

We assume that function $D$ satisfies the conditions
\begin{equation}
D(0)=0; 0<\lambda:=D'(0)<1; |D(x)|<|x|, \forall x \in(-a,a)\backslash\{0\}. 
\label{def:D}
\end{equation}

Consider the solution $g$ of Julia's equation
\eqref{jul1}, with $g=h/h'$, where $h$ satisfies Schr\"{o}der's equation (for similar results, see \cite{daniel,alva13,kucs2}).
Combining equations \eqref{eq:hlinhahj}, (\ref{eq:for1}), \eqref{eq62} and using that $g'(0)=1$, we obtain 
\begin{equation}
g(x)= \lim_{ n \rightarrow \infty} \frac{D^{n}(x)}{\prod_{j=0}^{n-1} D'(D^{j}(x))}.
\label{eqs:5-17b}
\end{equation}
Formula (\ref{eqs:5-17b}) guarantees that $g(0)=0$ can be rewritten as an 
infinite product, which is very useful for analysis and numerical calculations.
Let us define
\begin{equation}
x_k = D^{k}(x),\;\; k=0,1,\dots,\quad R_n=\frac{x_n}{\prod_{k=0}^{n-1}D'(x_k)}, \quad \rho_n = \rho(x_n),  
\label{eqs:5-20a}
\end{equation}
where $\rho(x) = \frac{D(x)}{D'(x)x}$. The sequence $x_k = D^{k}(x), k=0,1,\dots$ is the so-called splinter of $x=x_0$.
It follows that
\begin{equation}
\label{eqs:5-19}
R_n=\frac{D(x_{n-1})x_{n-1}}{D'(x_{n-1})x_{n-1}
	\prod_{k=0}^{n-2}D'(x_k)}
=\frac{D(x_{n-1})}{D'(x_{n-1})x_{n-1}}R_{n-1}= \rho_{n-1} R_{n-1}
\end{equation}
and \\
\begin{equation}\label{eqs:5-20}
g(x) = x \prod_{n=0}^\infty \rho(D^n(x)). 
\end{equation}

{

This representation of the solution as an infinite product is related to the solution of the functional equation
\begin{equation} \label{eq:g_til}
    {\rho(x)}\hat{g}(D(x)) = \hat{g}(x), 
\end{equation}
that can be obtained from equation \eqref{jul1} by setting $g(x) =  x \hat{g}(x)$ and using the condition $\hat{g}(0)=1$. We notice that applying the fixed-point iteration process
\begin{equation} \label{eq:g_til_fp}
     \hat{g}_{n+1}(x) = {\rho(x)}\hat{g}_n(D(x)), \quad n\geq 0,
\end{equation}
using, as an initial guess $\hat{g}_0$, a continuous function such that $\hat{g}_0(0)=1$ readily leads to formula \eqref{eqs:5-20}.   

\subsection{Solution of Julia's equation in the general case}
\label{sub:sfe}

In this section, we show how to solve Julia's functional equation in the general case. The assumptions used here correspond to the generic situation, in which the iteration function $D(x)$ is obtained from the solution of an ODE with a sufficiently regular field satisfying conditions for the existence and uniqueness of solutions.  

This method was presented in \cite{bedrikovetsky2002porous} for solving an inverse problem
and consists of two stages. First, the problem is subdivided into a finite (or at most countable) number of subproblems
in consecutive bounded intervals covering the total interval that
is the solution domain.
This subdivision is done so that
each subproblem  is easier to solve, since it satisfies an extra inequality 
$D(z) \ne z$.
The method for solving each of these ``reduced" subproblems is presented in Section~\ref{num1}.

Finally, we show how to piece together the solutions of the reduced 
subproblems to obtain the global solution of the functional equation.

We consider that the analysis in Section ~\ref{num1} corresponds to the case when $D(x)<x$, therefore
we set $x_{n+1}=D(x_n)$. In the opposite case, when $D(x)>x$,
we have to set $x_{n+1}=D^{-1}(x_n)$. This result is summarized in the following

\begin{lem} \label{lem:1}
	Let $D:[a,b) \rightarrow [a,\infty), \, 0 \leq a < b$ be a continuous monotone
	increasing function such that $D(a)=a$, possessing a continuous inverse. 
	Assume that $D(z) \ne z$ in $(a,b)$.
	Let $s_0$ be a point in $(a,b)$.
	Define the following sequences:
	
	(i) if $D(z) <z$ in $(a,b)$, let $s_{n+1}=D(s_n)$;
	
	(ii) if $D(z) >z$ in $(a,b)$, let $s_{n+1}=D^{-1}(s_n)$.

	Then the sequence $\{s_n \}$ is monotone decreasing and it converges to $a$.
	
\end{lem}

\begin{proof}
	
	We present the proof for case (i), as the proof for (ii) is analogous.
	
	Let
	$s_{n+1}=D(s_n) <s_n$. Since $s_{n+1}>a$,  $\{s_n\}$ is monotone decreasing and bounded from below,
	so it converges to $\bar{s}$ in $[a,b)$.
	From the continuity of $D$, the equation $s_{n+1}=D(s_n)$  implies that $\bar{s}=D(\bar{s})$,
	so $\bar{s}=a$.
	
\end{proof}

\begin{lem} \label{lem:2}
	Let $D:[a,b] \rightarrow [a,\infty), \, 0 \leq a <b$, be a continuous monotone
	increasing function possessing a continuous inverse, such that $D(a)=a$ and $D(b)=b$, but
	$D(z) \ne z$ in $(a,b)$.
	Let $s_0,r_0$ be any points in $(a,b)$.
	Define the following sequences: 
	
	(i) if $D(z) <z$ in $(a,b)$, $s_{n+1}=D(s_n)$, $r_{n+1}=D^{-1}(r_n)$; 
	
	(ii) if $D(z) >z$ in $(a,b)$, $s_{n+1}=D^{-1}(s_n)$, $r_{n+1}=D(r_n)$.

	Then the sequence $\{ s_n \}$ is monotone decreasing and it converges to $a$,
	and the sequence $\{ r_n \}$ is monotone increasing and it converges to $b$.
	
\end{lem}

\begin{proof}
	
	The proof follows from repeated applications of Lemma \ref{lem:1}.
	
\end{proof}

\begin{thm} \label{th:3}
	Let $D:[a,b) \rightarrow [a,\infty), \,  0 \leq a <b$ be a $C^1$ monotone
	increasing function possessing a $C^1$ inverse, such that $D(a)=a$, but
	$D(z) \ne z$ in $(a,b)$. 
	\textup{(} Here $[a,b)$ stands for $[a,b)$ or $[a,\infty)$. 
	\textup{)}
	
	Consider the set $G_a$ of continuous functions in $[a,\infty)$, differentiable at $a$, 
	with $g(a)=a$, $g'(a)=g'_a \ne 0$.
	Furthermore, consider the functional equation in $G_a$: 
	\begin{equation} \label{eq:funcg}
	g \big(D(z)\big)= D'(z)g(z) \quad \text{for any} \quad z \quad\text{in}\quad (a,b).
	\end{equation}
	
	Then the functional equation \eqref{eq:funcg} has a unique solution in $(a,b)$.
	This solution is proportional to $g'_a$.
	
\end{thm}

\begin{proof}
The proof is similar to the one in Proposition \ref{prop:1}  for the case that the fixed point is $x=0$.
\end{proof}

\begin{lem} \label{lem:4}
	
	Let $D:[a,b] \rightarrow [a,\infty), \, 0 \leq a < b$, be a $C^{1+\epsilon}$ monotone
	increasing function possessing a $C^1$ inverse,  such that $D(a)=a$ and $D(b)=b$, but
	$D(z) \ne z$ in $(a,b)$,
	and assume that the functional 
	equation \eqref{eq:funcg} is satisfied in $G_a \cap G_b$. 
	Furthermore, assume that $D'(a) >0$, $D'(b) >0$.
	
	Then $g'(b)$ is uniquely defined and it is proportional to $g'(a)$.
	
\end{lem}

\begin{proof}
We use Theorem \ref{th:3}. 
Let us prove only the case when $D(x)<x$ in $(a,b)$, since the other one is similar.
Take $r_0=s_0$ as any point in  $(a,b)$.
We compute $D(r_0)$ and $D(s_0)$ with the sequences $s_n \rightarrow a$,
$r_n \rightarrow b$ defined in Lemma~\ref{lem:2}; we relate $g'(a)$ to $g(s_0)$ and $g'(b)$ to $g(r_0)$
 using appropriate versions of formulas \eqref{eq62} and \eqref{eqs:5-17b}
(replacing $0$ by $a$, or $0$ by $b$ for cases (i) and (ii),
respectively).
Equating $g(r_0)=g(s_0)$, we obtain:
\begin{equation} \label{eqs:5-17c}
g'(a)=g(s_0) \bigg( \lim_{n \rightarrow \infty} \prod_{k=0}^{n-1} \frac{D'(s_k)}{s_n} \bigg)
=g'(b) \bigg( \lim_{n \rightarrow \infty} \frac{r_n}{\prod_{k=0}^{n-1} D'(r_k)} \bigg)
 \bigg( \lim_{n \rightarrow \infty} \prod_{k=0}^{\infty} \frac{D'(s_k)}{s_n} \bigg)
\end{equation}

\end{proof}

\begin{lem} \label{lem:5}
	
	Let $D:[0,c) \rightarrow [0,\infty)$ \textup{(}where $[0,c)$  may stand for $[0,\infty)$\textup{)}
	be a $C^1$ monotone increasing function possessing a $C^1$ inverse.
	Assume that $D(0)=0$, and that wherever $D(\bar{z})=\bar{z}$ then $D'(z)>0$,
	$D'(z) \ne 1$;

	Then the finite interval $[a,c)$ can be subdivided into a finite number of subintervals
	where $D(z) \ne z$,
	and the infinite interval into a countable number of subintervals,
	such that in the interior of each
	subinterval the quantity $D(z)-z$ does not change sign.
	
	In two consecutive subintervals separated by a fixed point where $D(\bar{z})=\bar{z}$,
	the values of $D(z)-z$ have opposite signs.
	
\end{lem}
}

\section{Properties of the solutions}
\label{secU}

\subsection{Regularity}

It is useful to study the regularity of the class of functions $g$ depending on the regularity of functions $D$. {
The construction $D\in C^{1+\epsilon}$ for some $\epsilon>0$  determines all continuous solutions $g$ admitting derivatives in} $0$ (there are other solutions that are 
only continuous). If $h$ is of class $C^k$, then $h'$ is of class $C^{k-1}$ and thus $g$ 
is of class $C^{k-1}$. Since $h\in C^{1}$, it is easy to prove (by induction in $s$) from 
the expression $g(x) = g'(0) h(x)/h'(x)$ that if $g\neq 0$ and $g\in 
C^{s}$, then $h\in C^{s+1}$. 

\begin{prop} \label{prop:2} \amf{Let} the functions $D$ and $h$ be as in Proposition~\ref{prop:1}\amf{,} with  $h$ defined by \eqref{eq:for1}. Then the function $h\in C^{k}$ with $k\geq 2$, if and only if $D\in C^{k}$.  
	\label{prop:5.2}
\end{prop}
\begin{proof} 
	If $h\in C^k$, then $D(x)=h^{-1}(\lambda h(x))$ is a composition of functions of class $C^k$ and thus $D\in C^k$.
	
	\amf{Conversely}, if $D\in C^k$ with $k\geq 2$, we may assume, by reducing the 
	interval \amf{radius} $a$ if necessary, that $D^{(j)}$ 
	is bounded in $(-a, a)$ for $1\le j\le k$. We have $\log (h'(x))=f(x)$ defined by 
	equation (\ref{eq:aux4}) and $h\in C^k$, if and only if, $\log h'\in C^{k-1}$, if and only 
	if, $h''/h' = (\log h')'\in C^{k-2}$. We have
	\begin{equation}
	f'(x) = \sum_{j=0}^{\infty} (\log D'(D^j(x)) - \log \lambda)' = 
	\sum_{j=0}^{\infty} \frac{D''(D^j(x))}{D'(D^j(x))} (D^j)'(x), 
	\end{equation}
	where this series of continuous functions converges uniformly in $(-a,a)$, 
	because $D^j(x)=h^{-1}(\lambda^j h(x))$, and so 
	$(D^j)'(x)=(h^{-1})'(\lambda^j h(x))\cdot \lambda^j \cdot h'(x)$, thus
	\begin{equation}
	\sum_{j=0}^{\infty} \frac{D''(D^j(x))}{D'(D^j(x))} 
	(D^j)'(x)=\sum_{j=0}^{\infty} 
	\frac{D''(D^j(x))}{D'(D^j(x))}\cdot \lambda^j \cdot (h^{-1})'(\lambda^j 
	h(x))\cdot h'(x).
	\end{equation} 
	Since $D''(D^j(x))$, $1/(D'(D^j(x)))$, $(h^{-1})'(\lambda^j h(x))$ and $h'(x)$ 
	are continuous functions which are uniformly bounded in $(-a,a)$, 
	and the series $\sum_{j=0}^{\infty}\lambda^j$ converges absolutely, it follows
	that $h\in C^{2}$. 
	
	We will show by induction on $r$, for $2\leq r \leq k$ that $h\in C^r$. Indeed, we prove that 
	$\sum_{j=0}^{\infty} (\log D'(D^j(x))-\log\lambda)^{(r-1)}$ can be written as
	\begin{equation}
	\sum_{j=0}^{\infty} \frac{\lambda^j P_r (h_1(x),h_2(x)
		,h_3(x),\lambda^j)}{(D'(D^j(x)))^{r-1}}
	\label{eq:aux5}
	\end{equation}
	where $P_r$ is a polynomial in $3r-1$ variables (which depends on $r$) with 
	\begin{equation*}
	h_1(x)=(D^{(s)} (D^j(x)))_{1\leq s\leq r},~~ h_2(x)=(h^{(s)}(x))_{1\leq s\leq r-1},~h_3(x)=((h^{-1})^{(s)}(\lambda^j h(x)))_{1\leq s\leq r-1}.
	\end{equation*}
	For example, the initial case $r=2$ follows from equation \eqref{eq:aux5}: we may take 
	\begin{equation*}
	P_2(u_1,u_2,v_1,w_1,z)=u_2\cdot v_1\cdot w_1
	\end{equation*}
	so we have 
	\begin{equation*}
	P_2(D'(D^j(x)), D''(D^j(x)), h'(x), (h^{-1})'(\lambda^j h(x)), \lambda^j)=D''(D^j(x))\cdot (h^{-1})'(\lambda^j h(x))\cdot h'(x).
	\end{equation*}
	By the induction hypothesis according to which $h\in C^{k-1}$, the functions $h_1(x)$, $h_2(x)$, $h_3(x)$ are continuous and uniformly bounded in $(-a,a)$.
	It follows that the series in equation (\ref{eq:aux5}) converges uniformly to 
	$(\log h')^{(r-1)}$, which is a continuous function (since it is given by a series of continuous
	functions which converges uniformly). The claim follows by induction using the following formula\amf{s:}
	\begin{align}
	& (D^{(s)}(D^j(x)))' = D^{(s+1)}(D^j(x))\cdot (D^j)'(x) = \nonumber \\
	& \;\;\;\; = D^{(s+1)}(D^j(x))\cdot (h^{-1})'(\lambda^j h(x))\cdot \lambda^j\cdot h'(x),\\
	& (h^{(s)}(x))'=h^{(s+1)}(x),\\
	& ((h^{-1})^{(s)}(\lambda^j h(x)))' = (h^{-1})^{(s+1)}(\lambda^j h(x))\cdot \lambda^j\cdot h'(x), \\
	& (\lambda^j)'=0
	\end{align}
	and
	\begin{align}
	& (D'(D^j(x))^{r-1})' =\\
	& = (r-1) D'(D^j(x))^{r-2}\cdot D''(D^j(x))\cdot (D^j)'(x) \\
	& = (r-1) D'(D^j(x))^{r-2}\cdot D''(D^j(x))\cdot (h^{-1})'(\lambda^j h(x))\cdot \lambda^j\cdot h'(x). 
	\end{align}
\end{proof}

\begin{rem}
	From Proposition~\ref{prop:5.2}, it follows that $g\in C^{1}$, if and only if, $D\in 
	C^{2}$. More generally, 
	$g\in C^{k}$, $k\geq 1$, if and only if, $D\in C^{k+1}$. It is possible to prove that 
	$g\not\equiv 0$ is 
	\amf{differentiable} at $0$, if and only if, $D''(0)$ \amf{exists}  .
	\label{rem10}
\end{rem}

\begin{rem}
	If $D$ is a real analytic function, then $h$ is also a real analytic function since $D$ 
	can be extended analytically to some disk $B\subset\mathbb{C}$ with center at the origin, 
	where 
	\begin{equation}
	\prod_{j=1}^{\infty} \frac{D'(D^j(x))}{\lambda}
	\end{equation}
	is the limit of a sequence of analytic functions in $B$ which converges uniformly. In this case, if $g$ \amf{is differentiable} at $0$ then $g$ is a real analytic function. Analyticity is important for the investigation of stability properties of the function $g$ depending on the function $D$, see \cite{alva13}.
\end{rem}

\subsection{Sufficient conditions for monotonicity}

Next, we present sufficient conditions that ensure the monotonicity of the solution of the functional equation (\ref{eq:func_eq}). To this end, for this we first introduce some important assumptions on the behaviour of function $D$.  

\begin{assumption}
	\label{ass:10} We assume that 
	$D(x)$ defined in (\ref{eq4}) is a $C^2$
	function for $0 \le x \le a$  satisfying
	\begin{equation} \label{eq:exist}
	 D(x)<x,\; 0<D^{\prime}(x)< d \; \text{\;for\;}\; 0
	\le x \le a;~~ D(0)=0 ~~\text{and}~~ D^{\prime\prime}(0) \ne 0,
	\end{equation}
	where $d < 1$ is a constant.
\end{assumption}
\begin{assumption}
	\label{ass:f1a} We consider that $D(x)$ defined in (\ref{eq4}) is a $C^2$ such that belongs to the class of functions  
	\begin{equation}
	\label{eq:welf1} \mathcal{M}=\{ D  \in C^{2}[0,a] : r_1
	<D(x) \le r_2,\; 0 \le r_3 \le D'(x) \le r_4,\;  D''(x) \le r_5 \},
	\end{equation}
	depending on certain constants $r_1,\cdots,r_5$. 
\end{assumption}

\begin{rem} 
	\label{th:existence}
	It is possible to check that, under  Assumptions $(\ref{ass:10})$,~$(\ref{ass:f1a})$, the
	solutions $g$ of the functional equation $(\ref{eq:func_eq})$ are
	uniformly bounded.  
\end{rem}

\begin{lem}
	\label{lem:finamono1} Suppose that function $D$  satisfies Assumptions \ref{ass:10}, \ref{ass:f1a} above and
	\begin{equation}
	\label{eq:monot}
	((D'(x))^{2}x-D(x)(D'(x)x))' > 0
	\end{equation}
	holds for all $x \in (0,a]$.
	Then the solution  $g$ of the functional equation (\ref{eq:func_eq}) is monotone increasing in $(0,a]$.
\end{lem}

\begin{proof}
	We set
	\begin{equation}
	\label{eq:fina1a}
	\rho(x)=D(x)/D'(x)x.
	\end{equation}
	Consider $y,x \in(0,a]$ \amf{with} $x < y$. Let $g$ be the solution of 
	$(\ref{eq:func_eq})$.
	Notice that from (\ref{eqs:5-20}) we have
	\begin{equation}
	\label{eq:fina2a}
	g(x)/g(y)=(x/y)\displaystyle  \prod_{j=1}^{\infty} (\rho(x_j)/\rho(y_j)),
	\end{equation}
	where  $x_j$ and $y_j$ represent the splinters of $x$ and $y$ as defined in (\ref{eqs:5-20a}).
	Notice that $\rho'(s)$ is a fraction with numerator equal to the left side of 
	inequality (\ref{eq:monot}) and the denominator equal to $(D'(x)x)^{2}$. 
	Then inequality (\ref{eq:monot}) yields $\rho'(s) > 0$ for all $s \in 
	(0,a]$.
	
	Using   that $x_j < y_j$ (because function $D$ is monotone increasing) and 
	that the function
	$\rho$ in \eqref{eq:fina1a} is monotone increasing, we have $\rho(x_j) < \rho(y_j)$ 
	for $j=1,2,\cdots$. As a consequence $g(x) < g(y)$; thus $g$ is monotone increasing in the interval $(0,a]$.
\end{proof}

\subsection{Sufficient conditions for superlinearity}

\amf{We now establish} sufficient conditions \amf{for} the solution $g$ of equation \eqref{eq:func_eq} to satisfy a superlinearity condition
\begin{lem}
	\label{lem:finamono1ab} 
	 Let be $D$ function satisfying Assumptions \ref{ass:10}, \ref{ass:f1a} and \amf{suppose} that
	\begin{equation}
	\label{eq:ass2ab}
	D''(0) < 0. 
	\end{equation} 
	Then the solution $g$ of the functional equation (\ref{eq:func_eq}) \amf{satisfies}
	$g(x) \geq x$ for all $x \in [0,a]$.
\end{lem}

\begin{proof}
	To prove this inequality, it is sufficient to check that the function 
	$s(x)=g(x)-x$ possesses a local minimum at $x=0$. Since we set $g'(0)=1$ , 
	then $s'(0)=0$. Notice that $s''(0)=g''(0)$ and $g(0)=0$. Using $g=h/h'$ and formulas 
	(\ref{eq:aux4}) and (\ref{eq:func1}) \amf{one may verify that}
	\begin{equation}
	g''(0)=\frac{D''(0)}{D'(0)} \displaystyle \lim_{n \rightarrow \infty} 
	\prod_{j=0}^{n-1} D'(D^{j}(0)).
	\end{equation}
	Since $D'(0)$ is positive and \amf{$D''(0) < 0$ we have} $g''(0)\leq 0$ and 
	the lemma follows.
\end{proof}

\begin{rem}
	It is possible to verify that, if $D''(0)=0$, then Lemma \ref{lem:finamono1ab} \amf{is} still valid
	assuming that \amf{the} first derivative such that $D^{(m)}(0) \neq 0$ is less than zero.
	\label{rem:prop56a}
\end{rem}

\section{Continuous dependence and stability}
\label{sec2}

Continuous dependence of the functional equation solution $g$ on
the given iteration function $D$ was established in
\cite{alva13, kucs1}. However, due to the relevance of this result for this paper, we present here a version of this result adapted to the new statement and another class takes place on the stability. 

We have the following lemma:

\begin{lem}
	\label{axm:dan-ama}	Let be $D_1$ and $D_2$ functions satisfying Assumptions \ref{ass:10} and \ref{ass:f1a}
	with $D_1(a) < D_2(a)$ (without loss of
	generality), there exists a data-independent constant $M$ such that the
	following
	inequalities hold: \\
	(i) $||D_1^{-1}-D_2^{-1}||_{\infty}$ $\le$ $M
	||D_1-D_2||_{\infty}$,
	\\
	(ii) $ |D_1^{'}(D_1^{-1}(s))-D_2^{'}(D_2^{-1}(s))|$  $\le$
	$||D_1^{'}-D_2^{'}||_{\infty}+
	M||D_2^{'}||_{\infty}||D_1-D_2||_{\infty}$, for all $s \in
	[0,D_1(a)]$.
\end{lem}

\begin{proof}
    (i) We have
\begin{equation}\label{eq:sama2}
D_1^{-1}(x)-D_1^{-1}(y)=
\int_{0}^{1}\frac{\partial }{\partial
	\alpha}(D_1^{-1}(\alpha x+(1-\alpha) y))d\alpha.
\end{equation}
Since $D_1 \in \mathcal{M}$ (see \eqref{eq:welf1}) then
$(D_1^{-1}(x))^{'}=1/D_1^{'}(x)< (1/r_1)$, now using \eqref{eq:sama2} we have
\begin{equation}\label{eq:sama3}
|D_1^{-1}(x)-D_1^{-1}(y)| \le
(1/r_1)|x-y|.
\end{equation}
Now, we assume that $x,y \in [0,D_1(B)] \subset [0,D_2(B)]$.
Fixed $x$, let $y=D_1(D_2^{-1}(x))$; it follows that
$D_1^{-1}(y)=D_2^{-1}(x)$ and $x=D_2(D_2^{-1}(x))$.
From (\ref{eq:sama3}) we have
\begin{eqnarray}\label{eq:sama4}
|D_1^{-1}(\tau)-D_2^{-1}(x)|=|D_1^{-1}(x)-D_1^{-1}(y)|
\le (1/r_1)|x-y| \le \nonumber
\\
(1/r_1) |D_1(D_2^{-1}(x))-D_2(D_2^{-1}(x))|.
\end{eqnarray}
From (\ref{eq:sama4}) we see that (i) holds. To prove (ii), notice that for $\Upsilon_1=D_1^{-1}(s)$ and
$\Upsilon_2=D_2^{-1}(s)$
\begin{eqnarray}\label{eq:sama5a}
|D_1^{'}(\Upsilon_1)-D_2^{'}(\Upsilon_2)| \le
|D_1^{'}(\Upsilon_1)-D_2^{'}(\Upsilon_1)|+
|D_2^{'}(\Upsilon_2)-D_2^{'}(\Upsilon_1)|
\end{eqnarray}
From (\ref{eq:sama5a}) and the mean value theorem, we obtain
\begin{equation}\label{eq:sama6}
|D_1^{'}(D_1^{-1}(s))-D_2^{'}(D_2^{-1}(s))| \le
||D_1^{'}-D_2^{'}||_{\infty}+||D_2^{''}||_{\infty}
|D_1^{-1}(s)-D_2^{-1}(s)|;
\end{equation}
using (i) in (\ref{eq:sama6}), we obtain (ii).
\end{proof}

We consider the functions $D$ defined on $[0,B]$ satisfying the
condition \eqref{eq:exist}. Taking $s=D(x)$, Eq.
(\ref{eq:func_eq}) can be rewritten as
\begin{equation}
\label{eqs:f4} g(s)=D'(D^{-1}(s))g(D^{-1}(s)).
\end{equation}
Now, we verify the validity of the following Lemma:
\begin{lem}
	\label{axm:lema2} Let us denote by $g_1$, $g_2$ the solutions of
	Eq. \eqref{eq:func_eq} 
	with corresponding $D_1$, $D_2$  satisfying assumption \ref{ass:10} and \ref{ass:f1a}
	with $D_1(a) < D_2(a)$ . Then there exist constants
	$v_1$, $v_2$, such that
	\begin{equation}
	\label{eq:eqa2o} \vert \vert g_1-g_2 ||_{\infty}  \le v_1
	||D_1^{'}-D_2^{'}||_{\infty}+
	v_2||D_1-D_2||_{\infty}.
	\end{equation}
\end{lem}

\begin{proof}
     Now, using \eqref{eqs:f4} and the notation
$\Upsilon_1=D_1^{-1}(s)$ and $\Upsilon_2=D_2^{-1}(s)$, we obtain
\begin{eqnarray}
\label{eqs:eqfin2an}
|g_1(s)-g_2(s)|=\vert
D'_1(\Upsilon_1)g_1(\Upsilon_1)-D'_2(\Upsilon_2)g_2(\Upsilon_2)
\vert \le
\nonumber
\\
\vert D'_1(\Upsilon_1)g_1(\Upsilon_1)-D'_2(\Upsilon_2)g_2(\Upsilon_1)\vert
+ \vert
D'_2(\Upsilon_2)(g_2(\Upsilon_2)-g_2(\Upsilon_1))\vert.
\end{eqnarray}
Using the mean value theorem and the definition of $\Upsilon_1$, $\Upsilon_2$, we have
\begin{eqnarray}
\label{eqs:fin2an1s1}
|D'_2(\Upsilon_2)(g_2(\Upsilon_2)-g_2(\Upsilon_1))\vert
\le
||D'_2||_{\infty}||g_2^{'}||_{\infty}||D_2^{-1}-D_1^{-1}||_{\infty}.
\end{eqnarray}
Moreover
\begin{eqnarray}
\label{eqs:eqfin2an1s2}
|D'_1(\Upsilon_1)g_1(\Upsilon_1)-D'_2(\Upsilon_2)g_2(\Upsilon_1)\vert
\le \hspace{85mm}\nonumber
\\
\hspace{10mm}
|g_1(\Upsilon_1)||D'_1(\Upsilon_1)-D'_2(\Upsilon_2)\vert
+|D'_2(\Upsilon_2)||g_2(\Upsilon_1)-g_1(\Upsilon_1)\vert \le
\hspace{15mm}\nonumber
\\
\hspace{5mm}
||g_1||_{\infty}|D'_1(\Upsilon_1)-D'_2(\Upsilon_2)\vert
+||D'_2||_{\infty}|(g_2-g_1)(\Upsilon_1)\vert.
\nonumber
\end{eqnarray}
Since $||D'_2||_{\infty} < d <1$, we have from
(\ref{eqs:eqfin2an}), (\ref{eqs:fin2an1s1}) that
\begin{equation}
\label{eqs:eqfin2an2s6}
(1-d)||g_1-g_2||_{\infty} \le ||D_2^{'}||_{\infty}||g_2^{'}||_{\infty}|
||D_2^{-1}-D_1^{-1}||_{\infty}+
||g_1||_{\infty}|D'_1(\Upsilon_1)-D'_2(\Upsilon_2)\vert
\end{equation}
Finally using Remark \ref{th:existence} and Lemma
 (\ref{axm:dan-ama}) in \eqref{eqs:eqfin2an2s6},
we obtain (\ref{eq:eqa2o}).		

\end{proof}

\section{Approximate methods for ODE field recovering}\label{sec1}

{
In this section we present different approximate methods to solve \textit{Problem 2}, i.e., recovering the field $v(x)$ of the ODE \eqref{eq1}. 

\subsection{Determining the function $D$ from the input data} \label{sec:recover_D}

The first step when applying one of the  strategies \eqref{eq:strat_1} or \eqref{eq:strat_2}, consists in the recovering of the iteration function $D$. With this goal in mind, in this section we discuss how to approximately obtain the function $D$ in \eqref{eq4} from the set of data points $\{(t_i,x_i),\;i=1,\dots,N \}$.  

{
We notice that if the input data corresponds to a solution $x(t)$ of an ODE, then we have that $x_i = x(t_i)$, $i=1,\dots,N$. Consequently, when the times $t_i$, $i=1,\dots,N$, are uniformly spaced  with a fixed stepsize $\Delta t$, then the function $D$ at points $x_i = x(t_i)$, $i=1,\cdots,N-1$ is given as $D(x_i)=x_{i+1}$. In this case, an approximation of function $D$ can be readily obtained by interpolation or a curve fitting procedure with the input data points $\{(x_i, x_{i+1}),\; i=1,\dots,N-1\}$.

More generally, when the discrete time points are not uniformly spaced, we first perform an interpolation in time to obtain an approximate  trajectory $x_{ap}(t)$ on the interval $[t_1,t_N]$. 
Finally, we approximate the function $D$ using interpolation or a curve fitting procedure with the input data $\{(x_i,x_{ap}(t_i+\Delta t)),\;i=1,\cdots,N-1\}$.

It is worth remarking that  both procedures can be readily adapted to the case where there are multiple sets of data available. 

{We remark that several methods for the identification of the iteration function $D$ in a discrete dynamical system are widely known (see, for instance \cite{brunton2016discovering,nelson2005nonlinear,roy2017dynamic,zhang2006discrete}) which could be  applicable here. However, in this paper we do not delve into that direction and use the procedures discussed above. 
}



}%

\subsection{Approximate methods to solve Julia's equation}

In this section we describe  several methods to approximate the solution $g(x)$ of the Julia's equation \eqref{jul1} satisfying the condition $g'(0)=1$. This allows us to reconstruct the ODE field \eqref{eq1} setting $v(x) = \log(D'(0))\,g(x)$.

\subsubsection{Infinite product approximation}
\label{ssec:inf_prof}




This method consists in the implementation of equation \eqref{eqs:5-20} and the algorithm presented below returns an approximation of $g$ at the given point $x=x_0$. We assume that the functions $D$ and $D'$ can be readily computed. 

\begin{algorithm}
	\caption{Implementation of formula \eqref{eqs:5-20} }\label{alg:inf_prod}
	\begin{algorithmic}[1]
		\REQUIRE $x_0$, $\epsilon$, functions $D$ and $D^{\prime}$
		\ENSURE $g(x_0) = q_n$
		\STATE $x_n=x_0$, error=1, lim=1, $q_n=1$ 
		\WHILE{ error $> \epsilon$}  
		\STATE  last=lim
		\STATE  $q_n=q_n  D(x_n)/(x_n D^{\prime}(x_n))$ 
		\STATE  $x_n=D(x_n)$ 
		\STATE   lim=$q_n$
		\STATE   error=$|$lim-last$|$/$|$last$|$
		\ENDWHILE
		\STATE $q_n = x_o q_n$
		\RETURN $q_n$
	\end{algorithmic}
\end{algorithm}

\begin{rem} \label{rem:disadv_inf}
    A disadvantage of this method is that it can only be used if the splinter corresponding to the initial point $x=x_0$ is well defined and converges to an attractive fixed point of $D$. This will not be the case, for instance, if the ODE has solutions that explode in finite time. However,  by implementing the strategy discussed in subsection \ref{sub:sfe} that divides the domain of interest in appropriately chosen subintervals,  this method can be applied under very generic conditions.
\end{rem}
 

\subsubsection{Fixed point approximation}
\label{ssec:fixed_p}

This method consists in the implementation of the fixed-point iteration given by equation \eqref{eq:g_til_fp}. Even though the  infinite product and the fixed-point approximations are equivalent, we introduce a method based on the later approximation that relies on interpolation in order to avoid the direct computation of splinters, which is explicitly used in the former approximation.

\begin{algorithm}
	\caption{Implementation of formula \eqref{eq:g_til_fp} }\label{alg:fixed_p}
	\begin{algorithmic}[1]
		\REQUIRE $x_0,\dots,x_m$, $\epsilon$, functions $D$ and $D^{\prime}$
		\ENSURE $g(x_0) = g_0, \dots, g(x_m) = g_m$
		\STATE  $g_0 = \cdots = g_m = 1$, error=1, lim=1
		\STATE  $q_0 = D(x_0)/(x_0 D^{\prime}(x_0)), \dots, q_m = D(x_m)/(x_m D^{\prime}(x_m))$
		\WHILE{ error $> \epsilon$}  
		\STATE  $gl_0 = g_0, \dots, gl_m = g_m$
		\STATE Compute function $g(x)$ interpolating data : $(x_0,g_0), \dots, (x_m,g_m)$
		\STATE  $g_0= q_0 g(D(x_0)), \dots, g_m= q_m g(D(x_m))$ 
		\STATE   error=max($|gl_j-g_j|$)/max($|gl_j|$)
		\ENDWHILE
		\STATE $g_0 = x_o g_0, \dots, g_m = x_m g_m$
		\RETURN $g_0, \dots, g_m$
	\end{algorithmic}
\end{algorithm}

\begin{rem} \label{rem:fixed_p}
Since this method does not rely on splinter computation, it can be readily used in cases in which the splinter corresponding to a point $x$ in the domain of interest is not defined and there is no need for dividing the domain of interest in subintervals. This represents the great advantage of this method in comparison with the method discussed in Section \ref{ssec:inf_prof}. 

Another useful characteristic of this method is its flexibility regarding the interpolation step (fifth step of algorithm \ref{alg:fixed_p}), which allows the final user to choose an interpolation method based on its own criteria.  
\end{rem}

\subsubsection{Least square approximation}
\label{asyn1a}

This method relies on the assumption that we have a parametrization of the unknown field $v(x)$ of the ODE \eqref{eq1}, i.e., we consider that $v(x) = v_p(x)$ where $p$ represents the parameter vector. The goal is to estimate the parameter vector $p^\ast$ associated with the given data, leading in general to a data fitting problem. Optimization techniques have been extensively used for the estimation of parameters in  ordinary and partial differential equations (see for instance \cite{hofmann2018new, lu2004estimation, muller2004parameter, peifer2007parameter, muller2002fitting}), and the proposed  method also follows this approach. 

More specifically, by taking into account the relationship between the field $v(x)$  and the solution of Julia's equation $g(x)$ in \eqref{eqref1a}, we have the approximation $v(x) \approx v_{p^\ast}(x)$ where $p^\ast$ represents the solution of the following optimization problem. 

Find
\begin{equation}\label{eq:ls_method}
p^\ast\quad\text{minimizes}\quad ||\mathbf{R}(p)||^2 
\end{equation}
subject to the constraints: $v_p(0)=0$, $v'_p(0)=\log(D'(0))$, where
\begin{equation*}
    || \mathbf{R}(p) ||^2 = \sum_{j=1}^N | R_j(p) |^2,\quad R_j(p) = v_p(D(x_j))-D'(x_j)v_p(x_j).
\end{equation*}

\begin{rem}\label{rem:l_square}
The optimization problem \eqref{eq:ls_method} leads to a system of linear equations when the dependence of $v_p$ on the parameters $p$ is linear and whenever this dependence is nonlinear the corresponding optimization problem is also nonlinear.
\end{rem}
}

\section{Numerical experiments} \label{sec:examples}

In this section, we present  numerical examples illustrating the application of the approximate methods discussed in the previous section. In these examples, given the function $D(x)$, we approximately recover the field $v(x)$. Moreover, since the exact solutions for these examples are known, they allow us to verify the robustness,  advantages and drawbacks of the proposed methods.  

\subsection{Example 1: recovering a quadratic field}\label{example1}

For $0<a<1$, we consider the quadratic field $v(x)= \log(a) x (1-x)$.
The corresponding iteration function is given by
\begin{equation} \label{eq:quad_D}
    D(x)=\frac{ax}{1-(1-a)x},
\end{equation}
for $x<x_s=1/(1-a)$. This bound is related to the fact that for $x>1$ the solution of the associated ODE explodes in finite time.

We let $a=0.5$ and generate a synthetic set of data points  with different degree of randomness $\sigma$. We take $\{(x_i,y_i), i=0,\dots,N\}$ where $y_i = D(x_i) + \sigma_i$ and $\sigma_i$ are independent (pseudo)random numbers uniformly distributed in the interval $(-\sigma/2, \sigma/2)$.

From these data we estimate the iteration function $D$ by a curve fitting method, considering the dependence of $D$ on the parameter $a$ given in equation \eqref{eq:quad_D}. The resulting iteration function for the case of $\sigma = 0.5$ is given in Figure \ref{figomega}. Next, we recover the corresponding field $v(x)$ using  Algorithm \ref{alg:inf_prod}.  
{ In  Figure \ref{figomega1} the exact and recovered fields for several values of the parameter, i.e., $\sigma=0.1,0.5,0.9,1.9,2.5$ are presented. Notice that the exact field  
	differs slightly from the approximated field for values of 
	$\sigma$ less than one. Despite increasing the difference between both fields for higher values of  $\sigma$, the exact field is recovered in a stable and accurate way. This behavior suggests that the recovery method works well when used to simulate experimental data, which would be contaminated with errors.}



\begin{figure}[htp]
	\begin{center}
		\includegraphics[scale=0.4]{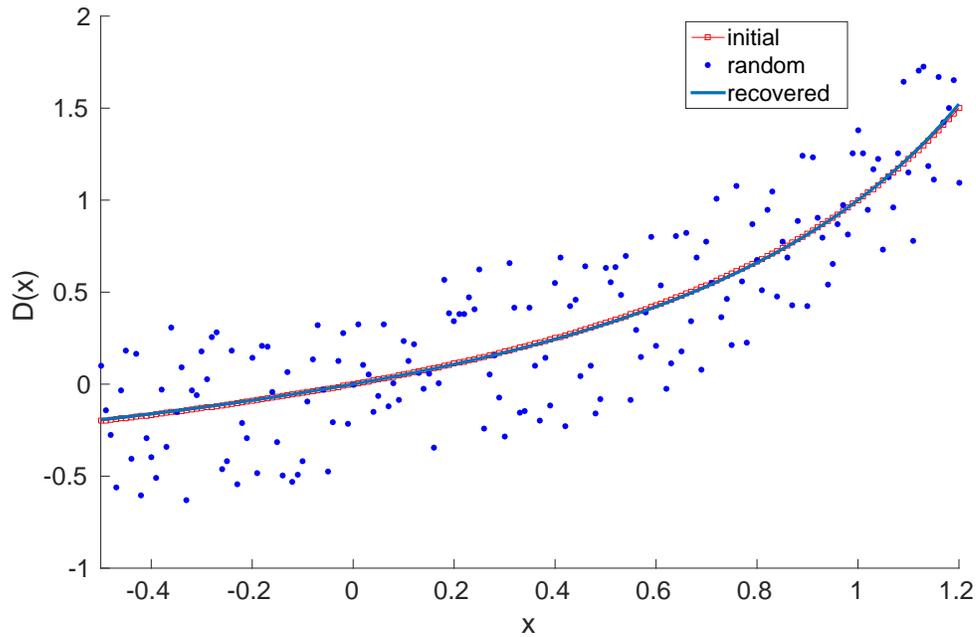}
	\end{center}
	\caption{
	{Iteration function $D$ (solid line), points of the data set with a noise level  $\sigma=0.5$ (blue points) and the recovered iteration function (red circle) as discussed in Example 1 (subsection \ref{example1}). Notice that the exact and recovered iteration functions are basically indistinguishable.}}
	\label{figomega}
\end{figure}


We observe that the field $v(x)$ is recovered beyond the singular point  $x_s=1/(1-a)$, since the estimated iteration function approximates the correctly extended version of the exact iteration function \eqref{eq:quad_D} (which captures a kind of continuation from the infinity of the trajectories that explode in finite time). 

In Table \ref{tab:table1} we show the relative errors   $\epsilon_{D}$, $\epsilon_{D'}$ and $\epsilon_{v}$ corresponding to the approximations of functions $D$, $D^{'}$ and the field $v$, respectively. In the last column we show the values of the stability constant $C_v=\epsilon_{v}/(\epsilon_{D}+\epsilon_{D^{'}})$. The results in this column indicate that the method is robust and stable since  the variables $C_v$ remain uniformly bounded as $\sigma$ is varying in the range $0.1$--$4.5$.

\begin{figure}[htp]
	\begin{center}
		\includegraphics[scale=0.28]{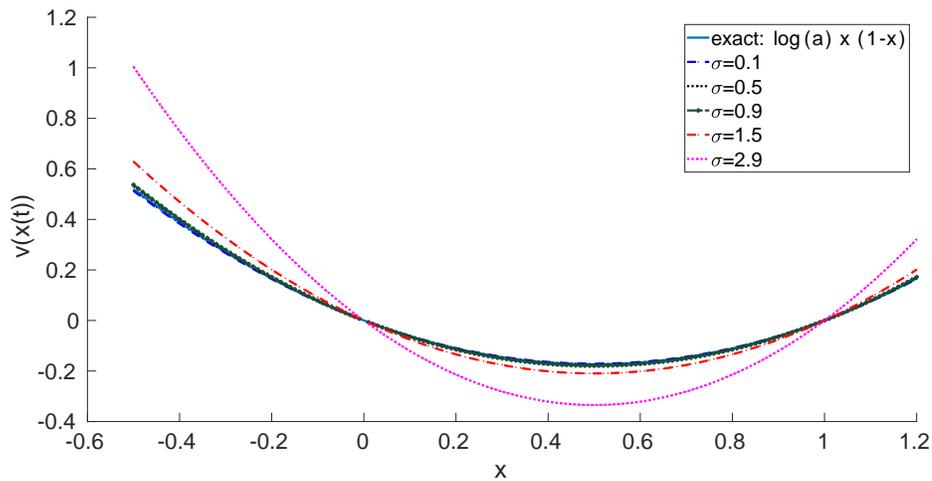}
	\end{center}
	\caption{
	{The exact field  $v(x)$ (solid blue line) and the recovered field corresponding to Example 1 (subsection \ref{example1}), for different values of $\sigma$.
			For $\sigma=0.1$ (dashed blue line), $\sigma=0.5$ (dotted black ), $\sigma=0.9$ (dash-dot green line ), $\sigma=1.5$ (dashed red blue) and $\sigma=2.9$ (dotted magenta line). Notice that for $\sigma<1$  exact and recovered fields are  indistinguishable.}	
	}
	\label{figomega1}
\end{figure}

\begin{table}[htp]
\begin{center}
\caption{Values of the relative errors for the iteration function $D$, its derivative $D'$ and the field $v$, and the stability constant $C_v$ corresponding to example \ref{example1}.}
\label{tab:table1}
 \begin{tabular}{|c | c| c| c| c| c|} 
 \hline
 $\sigma$ & $\epsilon_{D}$ & $\epsilon_{D^{'}}$ & $\epsilon_{v}$ & $C_v$ \\ [0.5ex] 
 \hline
 0.1 & 2.51 & 0.026 & 0.019 & 0.0075 \\ 
 \hline
 0.5 & 2.46 & 0.0883 & 0.06 & 0.026 \\
 \hline
 0.9 & 2.65 & 0.45 & 0.25 & 0.080  \\
 \hline
 1.9 & 2.3 & 0.67 & 0.49 & 0.1661 \\
 \hline
 2.9 & 2.77 & 0.822 & 0.38 & 0.107 \\ 
 \hline
  3.9 & 2.38 & 0.3137 & 0.27 & 0.10 \\ 
  \hline
   4.5 & 2.59 & 0.263 & 0.1614 & 0.05 \\ 
 \hline
 \hline
\end{tabular}
\end{center}
\end{table}

\subsection{Example 2: recovering a cubic field}\label{example2}

For $0<a<1$, we consider the cubic field $v(x)= \log(a) x (1-x^2)$.
The corresponding iteration function is given by
\begin{equation} \label{eq:cubic_D}
    D(x)=\frac{ax}{\sqrt{1+(a^2-1)x^2} },
\end{equation}
for $|x|<x_s=1/\sqrt{1-a^2}$. As in Example 1, this bound is a consequence of finite time blow up of the solutions of the associated ODE, for $|x|>1$. However, in comparison with the previous example, this iteration function cannot be extended in an appropriate way for $|x|\geq x_s$, therefore  a straightforward application of Algorithm \ref{alg:inf_prod} for $|x|>1$ is not possible in this example. 

{
\begin{figure}[htp]
	\begin{center}
		\includegraphics[scale=0.8]{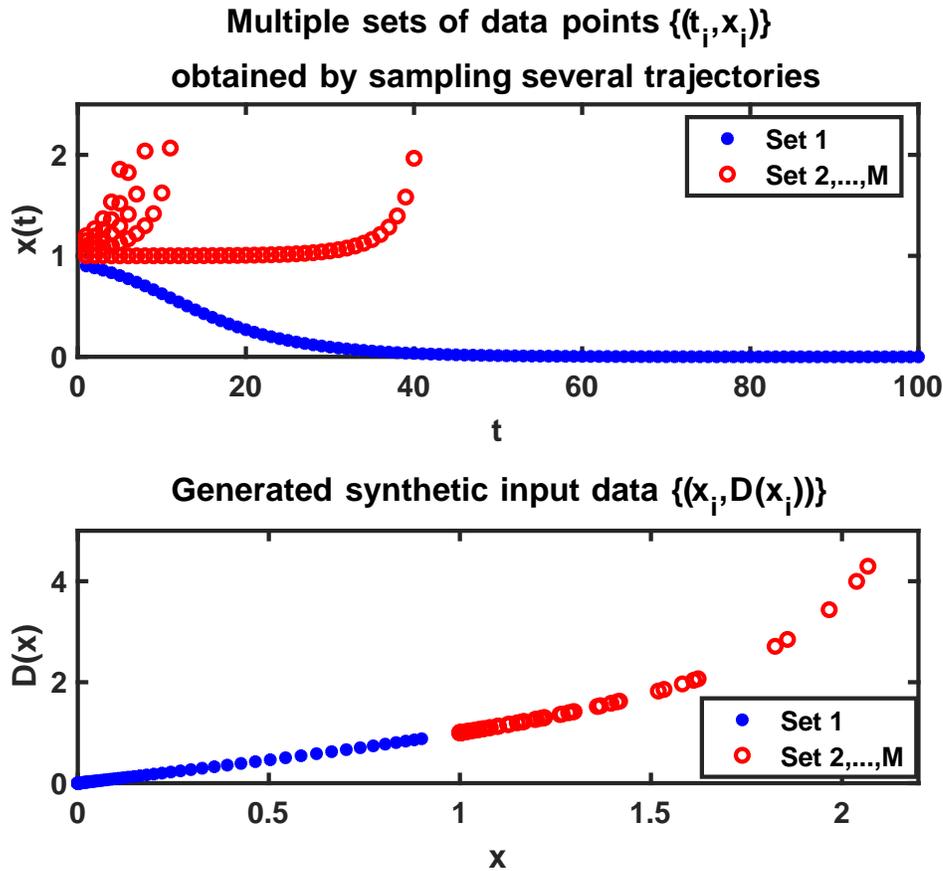}
	\end{center}
	\caption{Multiple sets of data points (upper plot) and the corresponding synthetic data (lower plot) used in Example 2 (subsection \ref{example2}). The set 1 (blue points) is used in both cases (a) and (b), whereas the other sets (red circles) are only used in case (b). }
	\label{fig:ex2_data}
\end{figure}

In the following numerical illustrations, we consider $a=0.9$ and use a synthetic input data set  $\{(x_i,D(x_i)), i=1,\dots,N\}$.  
First, we approximate the iteration function $D$ using the adaptive Antoulas-Anderson (AAA) algorithm for rational approximation introduced in \cite{Nakatsukasa_2018} and afterwards, we recover the field $v(x)$ using Algorithm \ref{alg:fixed_p} by applying the  AAA algorithm in the interpolation step.


\textbf{Case (a):} A set of data points $\{(t_i, x_i)\}$ is obtained by sampling a trajectory lying in the interval $(0,1)$. This set of data points is shown in the upper plot of Figure \ref{fig:ex2_data} as Set 1 (blue points). The synthetic input  data $\{(x_i,D(x_i)), i=1,\dots,N\}$ is generated from this single set of data points, as indicated in subsection \ref{sec:recover_D}. This data is shown in the lower plot of Figure \ref{fig:ex2_data}. 

  In Figure \ref{fig:ex6_a-f1}, we present the graphics corresponding to the approximation of the iteration function as well as its derivative in the interval $(0,1)$. For both functions the absolute error is less than $2.0\times10^{-10}$. The higher errors occur in the neighborhood of $x=1$, where the input data is more sparse.
 
\begin{figure}[htp]
	\begin{center}
		\includegraphics[scale=0.8]{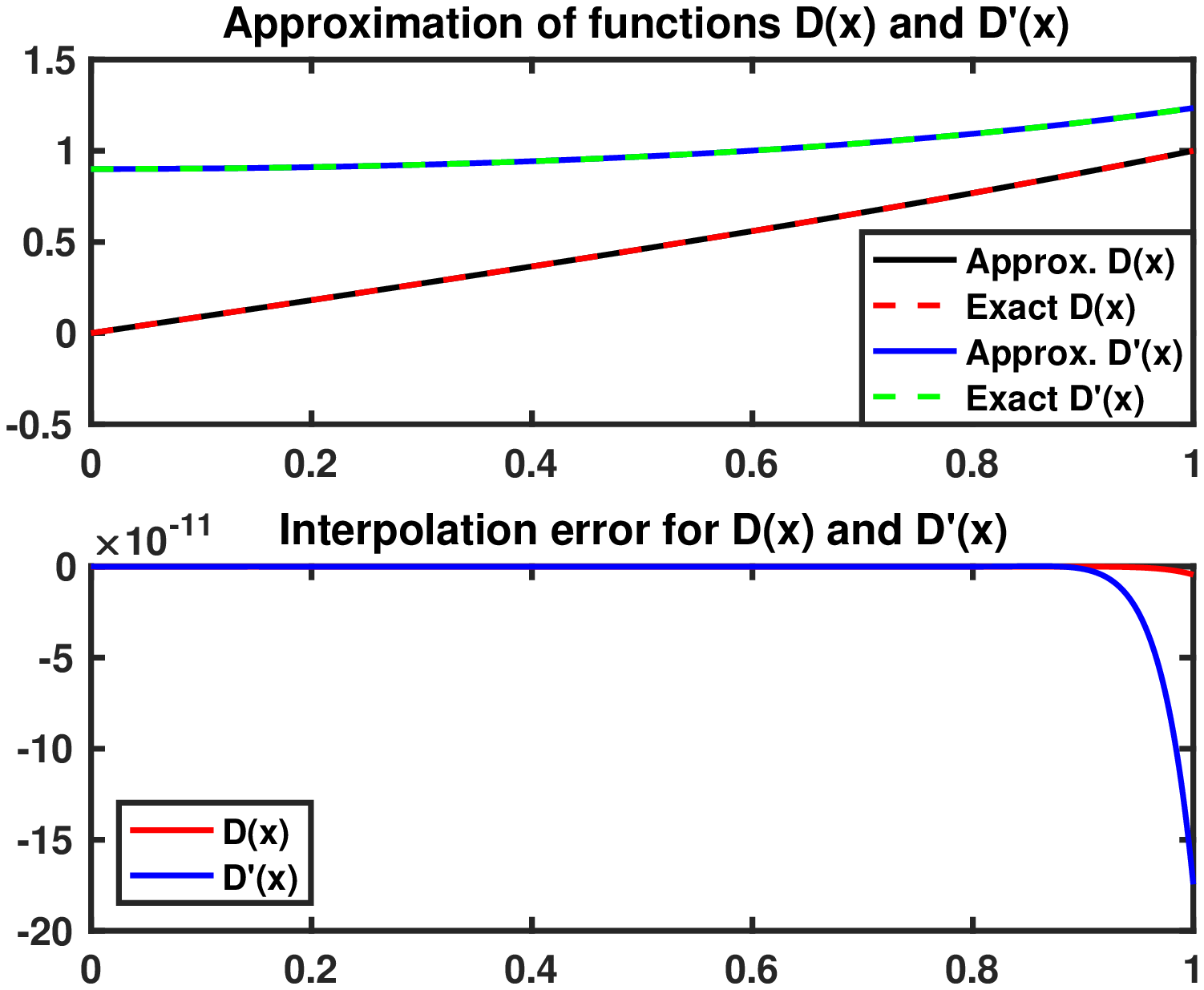}
	\end{center}
	\caption{Exact and approximate iteration function and its derivative (upper plot) and the associated approximation errors (lower plot) in the interval $(0,1)$, corresponding to case (a) of Example 2 (subsection \ref{example2}). Notice that the exact functions and their approximations are  indistinguishable. }
	\label{fig:ex6_a-f1}
\end{figure}

In Figure \ref{fig:ex6_a-f2}, we show the approximated field $v(x)$. The absolute error is less than $2.5\times10^{-6}$, illustrating a very accurate recovering  of the field. As expected the higher errors also occur in the neighborhood of $x=1$. 

Finally, by fitting a third degree polynomial  to the approximate values of the field, we get the following approximation
$$ v(x)\approx 0.10536 \,( x^3 + 2.0352\times10^{-10} \,x^2 - x -1.6465\times10^{-17} ),$$
whose coefficients have an overall absolute error of at least $10^{-6}$.  Consequently, this gives an approximation of $v(x)$ that can be used beyond the interval $(0,1)$.

\begin{figure}[htp]
	\begin{center}
		\includegraphics[scale=0.8]{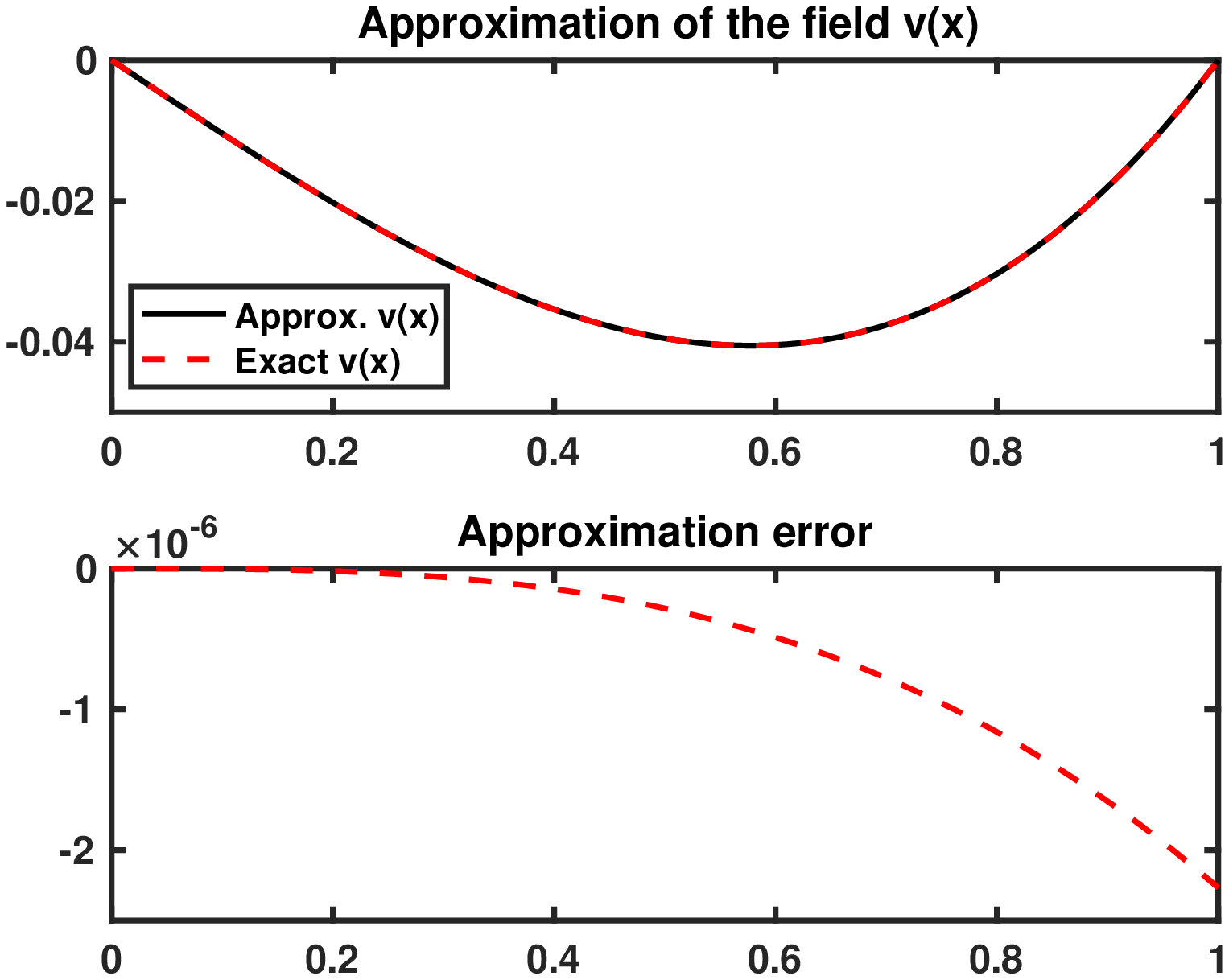}
	\end{center}
	\caption{Exact and approximated field $v(x)$ (upper plot) and  approximation error (lower plot)  in the interval $(0,1)$ corresponding to case (a) of Example 2 (subsection \ref{example2}).  Observe that the exact field and its approximation are indistinguishable. }
	\label{fig:ex6_a-f2}
\end{figure}

\textbf{Case (b):} In order to generate the input data, we proceed as in the previous case but using multiple sets of data points.  We construct these sets by sampling several trajectories lying in  the  interval $(-x_s+0.25, x_s-0.25)\approx (-2.04,2.04)$. They contain all sets of data points shown in the upper plot of Figure \ref{fig:ex2_data} and also those obtained by symmetry with respect to the axis of abscissas. The synthetic input data $\{(x_i,D(x_i)), i=1,\dots,N\}$ is generated from this multiple sets of data points, as indicated in subsection \ref{sec:recover_D}. In the lower plot of Figure \ref{fig:ex2_data} we show half of this synthetic data, since the whole set of data is symmetric with respect to the origin.


In Figure \ref{fig:ex6_12a}, we show the graphics corresponding to the approximation of the iteration function $D(x)$ as well as its derivative in the interval $(-2.04,2.04)$. For both functions the absolute error is less than $6.0\times10^{-11}$; however, the negative impact of the singularities of these functions for $x=\pm x_s\approx \pm 2.29$ is already noticeable in the neighborhood of $x=\pm 2.04$.

\begin{figure}[htp]
	\begin{center}
		\includegraphics[scale=0.8]{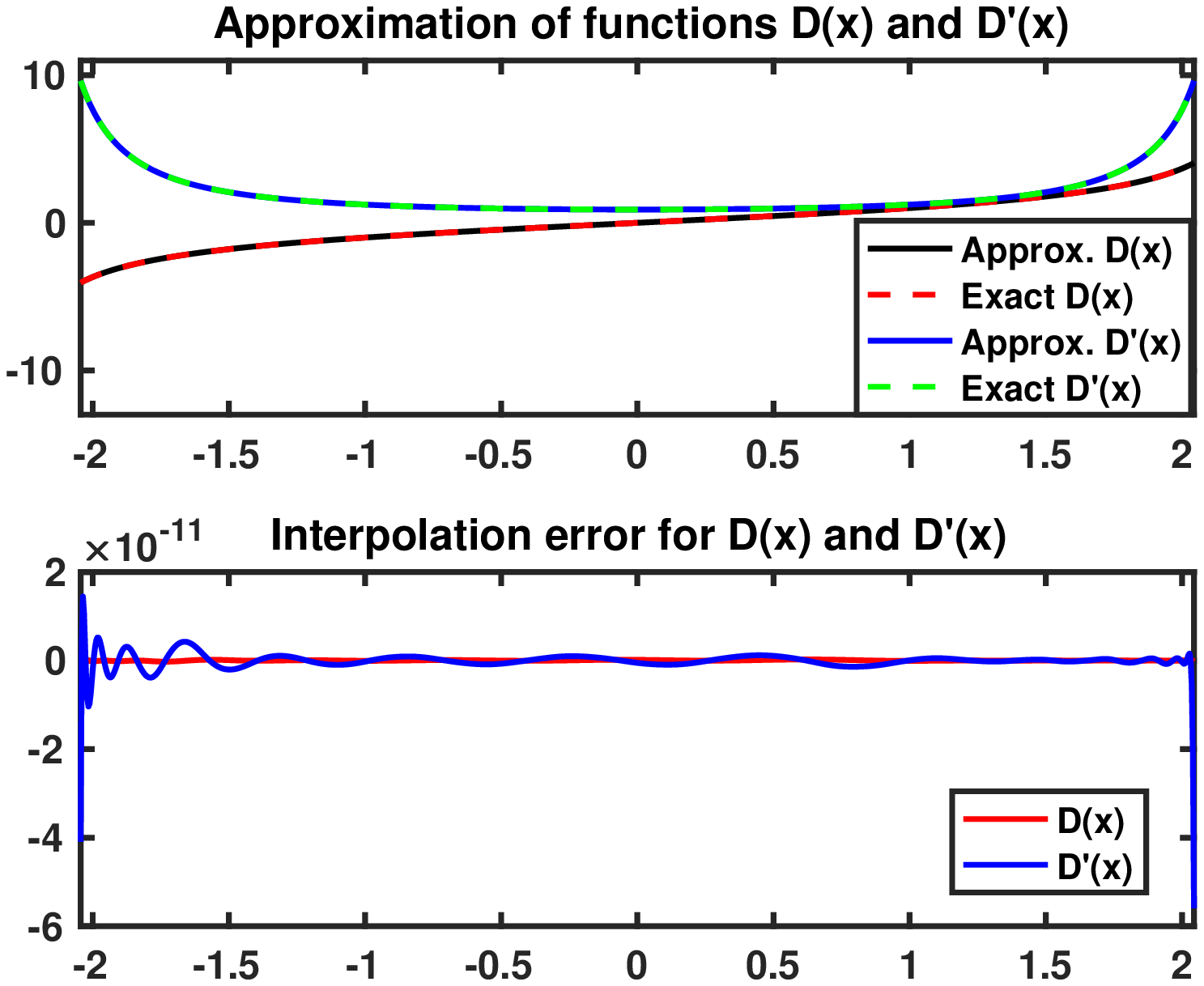}
	\end{center}
	\caption{Exact and approximated iteration function and its derivative (upper plot) and the associated approximation errors (lower plot) corresponding to case (b) of Example 2 (subsection \ref{example2}).  Notice that the exact functions and their approximations are almost indistinguishable. }
	\label{fig:ex6_12a}
\end{figure}

In Figure \ref{fig:ex6_3a}, we show the field $v(x)$ and its approximation. The absolute error is less than $2.0\times10^{-5}$, illustrating a very accurate recovering  of the field. The higher errors also occurs in the neighborhood of $x=\pm 2.04$.

Performing a curve fitting procedure with a third degree polynomial, we get the following approximation
$$ v(x)\approx 0.10536 \,( x^3 - 1.0620\times10^{-10} \,x^2 - x + 8.6391\times10^{-11} ),$$
whose coefficients have an overall absolute error of at least $10^{-6}$. Therefore, this approximation of $v(x)$ can give accurate results beyond the interval $(-x_s,x_s)\approx (-2.29,2.29)$. 



\begin{figure}[htp]
	\begin{center}
		\includegraphics[scale=0.8]{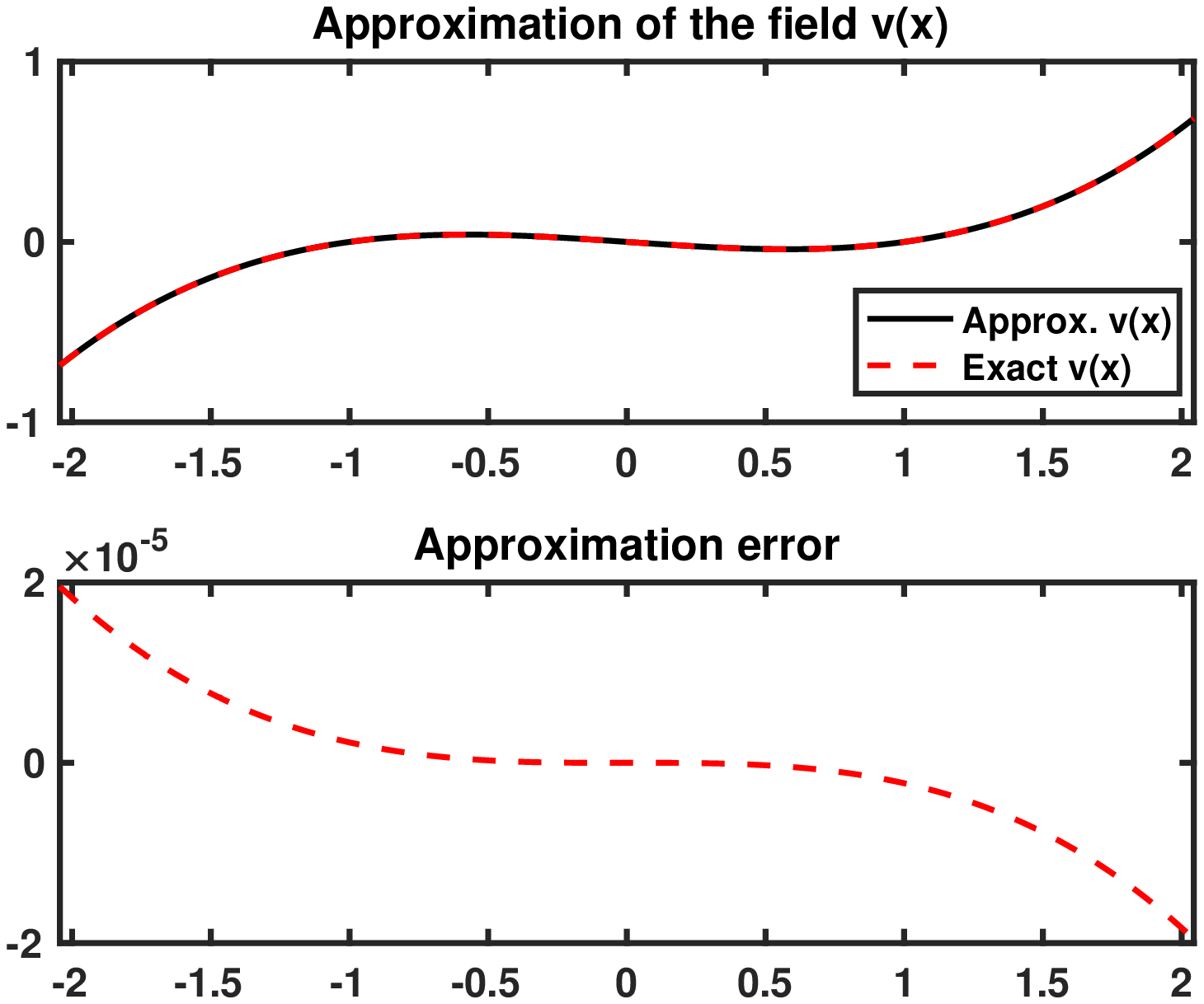}
	\end{center}
	\caption{Exact and approximated field $v(x)$ (upper plot) and  approximation error (lower plot)  in the interval $(-2.04,2.04)$ corresponding to case (b) of Example 2 (subsection \ref{example2}).   Notice that the exact field and its approximation are   indistinguishable. }
	\label{fig:ex6_3a}
\end{figure}
}

As a remark we should note  that the interval where we are able to accurately recover the field $v(x)$ using Algorithm \ref{alg:fixed_p} is a closed subinterval  of $(-x_s,x_s)$. However, we  accurately extrapolated the approximation beyond this interval using a curve fitting procedure by taking into account a parametric dependence of the field.
In general,  if we don't have any additional information, in order to  recover the field in a larger interval, we need to use an iteration function corresponding to a smaller $\Delta t$, i.e. to use a data set obtained with a higher sampling rate of the trajectories. 

\subsection{Example 3: recovering a field with a singular fixed point}\label{example3}

For $0<a<1$, we consider the field $v(x)= \log(a) (x+1/2)\log (2x+1)$ for $x\geq -1/2$. The corresponding iteration function is given by
\begin{equation} \label{eq:sing_D}
    D(x)=\frac{(2x+1)^a-1}{2},
\end{equation}
that besides the regular fixed point $x=0$ also has a singular fixed point at $x_s =-1/2$. Notice that at $x_s$ both functions  $v(x)$ and $D(x)$ are not differentiable.

We let $a=0.5$ and generate a synthetic set of data points $\{(x_i,y_i), i=1,\dots,N\}$ in the interval $(-0.5, 1.2)$  where $y_i = D(x_i)(1 + \sigma_i)$ and $\sigma_i$ are independent (pseudo)random numbers uniformly distributed in the interval $(-\sigma, \sigma)$. As in the previous examples the data $\{(x_i,D(x_i)), i=1,\dots,N\}$ is obtained from a multiple set of data points $\{(t_i,x_i)\}$ following the procedure discussed in subsection \ref{sec:recover_D}. We consider different values of $\sigma$ (up to a $5\%$ perturbation), and obtain a  rational approximation of the iteration function $D(x)$ using the full-Newton least square algorithm presented in \cite{etna_vol35_pp57-68}. Graphics of the approximate iteration function, its derivative and the approximation errors corresponding to the case with $\sigma=0.05$ ($5\%$ perturbation) are shown in Figure \ref{fig:ex2_12}. For the iteration function the absolute error is less than $10^{-2}$. However, the absolute error for the derivative is close to $0.5$, and as expected the worst approximation occurs near the singularity.

\begin{figure}[htp]
	\begin{center}
		\includegraphics[scale=0.55]{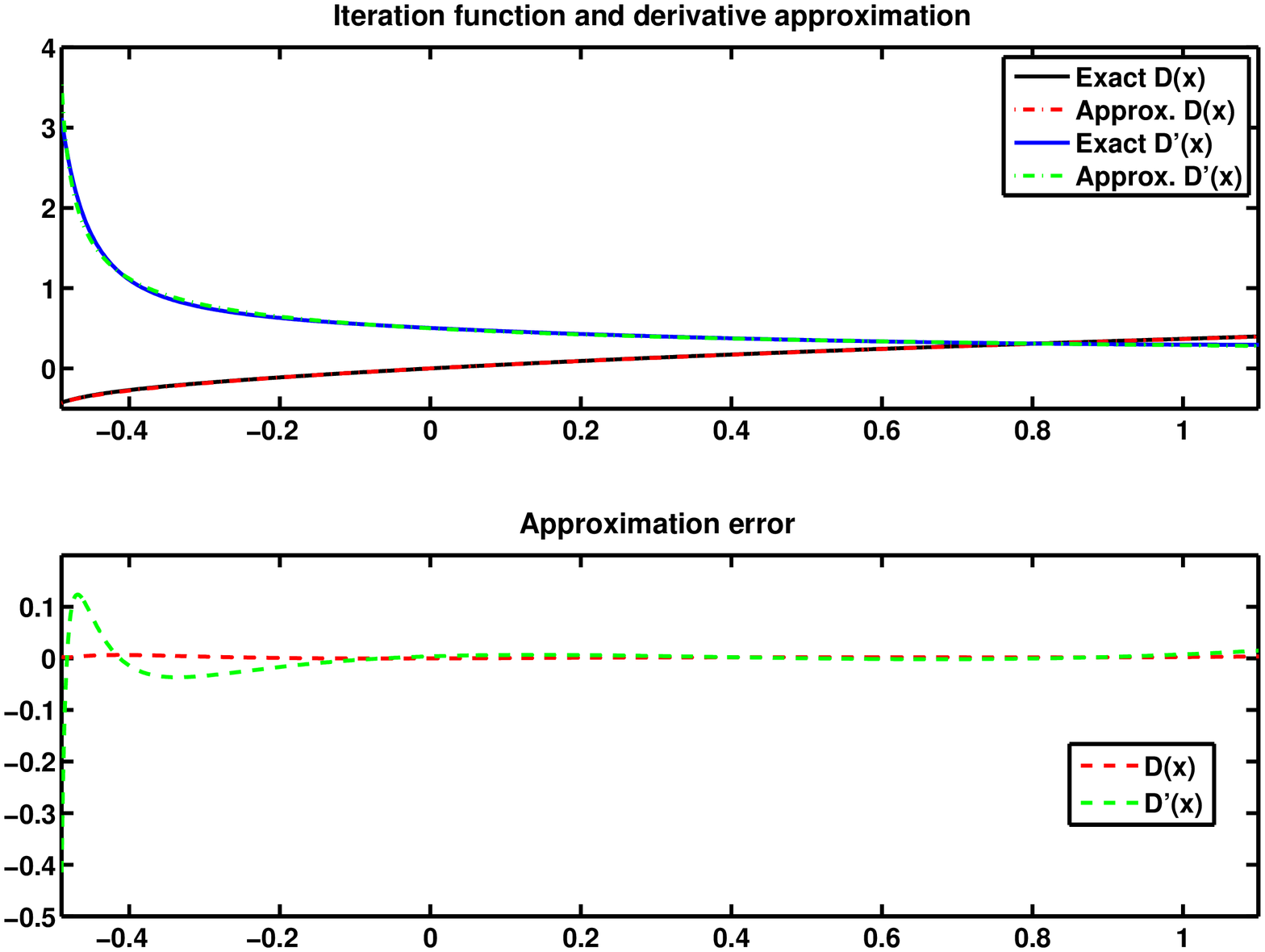}
	\end{center}
	\caption{Exact and approximated iteration function and its derivative (upper plot) and the corresponding approximation errors for the iteration function and its derivative (lower plot) corresponding to Example 3 (subsection \ref{example3}).   }
	\label{fig:ex2_12}
\end{figure}

In Figure \ref{fig:ex2_3}, we show the field $v(x)$ recovered using a least square rational approximation, also based on the algorithm mentioned above.  The absolute error obtained with this method is less than $8.0\times10^{-2}$. The result is not as accurate as in the previous example, but the current example has a different type of singularity that affects the derivatives of the iteration function and the field. Moreover, in the approximation of the field  the higher errors occurs near $x=1.2$, indicating that the effect of the sparseness of the data around this point has a stronger impact than the singularity.

\begin{figure}[htp]
	\begin{center}
		\includegraphics[scale=0.55]{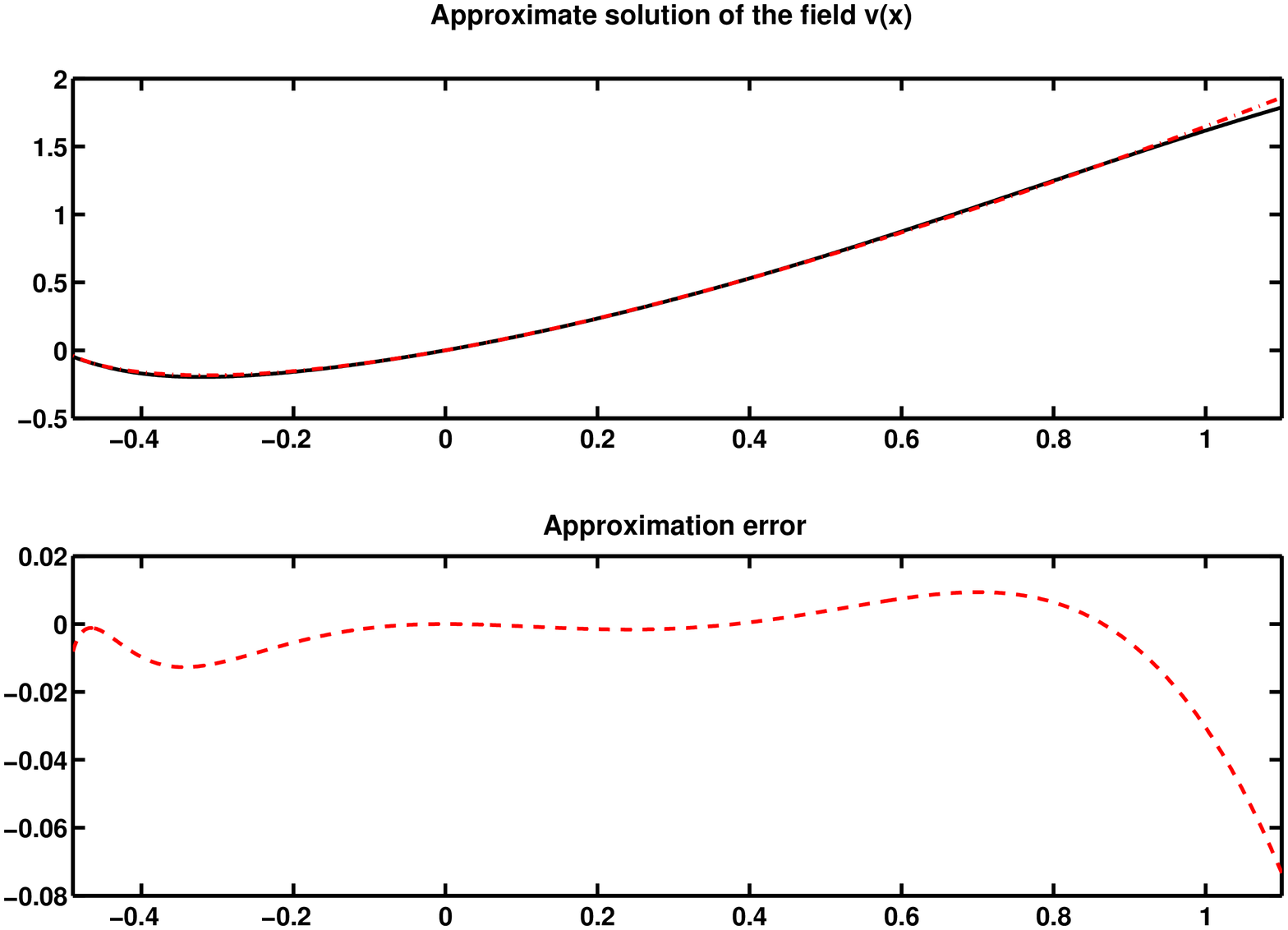}
	\end{center}
	\caption{Exact and approximated field (upper plot) and the approximation error (lower plot) corresponding to Example 3 (subsection \ref{example3}).  Notice that the different between the field and its approximation is only noticeable around $x=1.2$. }
	\label{fig:ex2_3}
\end{figure}

\section{Final remarks}\label{rec:final_r}

In this work, we present a complete description for the inverse problem of the determination of an ODE based on solution values. Conditions of existence, uniqueness and stability of this inverse problem were set for a broad set of functions that allows for practical uses. Here the one-dimensional case has been studied, leaving for future work the case of higher dimensions.

The numerical examples illustrate that the proposed approximate methods are quite robust and can be applied in a wide variety of cases. We explore the close relationship between ODEs and the solution of Julia's equation. The proposed method constitutes an alternative algorithm for the estimation of parameters for ODEs.

\noindent
\textbf{Acknowledgments}

	The authors thank Prof. Dan Marchesin for introducing us to the topic. The second author acknowledges Eng. Ely Maranh\~ao for her helpful comments and motivation of the results. We also are greatful for the collaboration of  Luzia Maranh\~ao, Lia Vinhas de Quiroga, and Hector Quiroga Zambrana.
	
	The second author's work was partially supported by IMPA/CAPES.
	The third author was partially supported  by FAPEMIG under Grant APQ 01377/15. 
	The fourth author  was partially supported by DICYT grant 041933GM from VRIDEI-USACH.

\section*{References}
\bibliographystyle{plain}
\bibliography{inverproble}

\end{document}